\documentclass[11pt,reqno]{amsart}
\usepackage{mathtools,amsthm,amssymb,amsfonts,wasysym}
\usepackage{graphicx}
\usepackage{hyperref}
\usepackage{verbatim}
\usepackage{csquotes}
\usepackage{geometry}
\usepackage{color}
\usepackage{tikz-cd}
\usetikzlibrary{cd}
\usepackage{caption}
\usepackage{subcaption}
\usepackage{setspace}
\usepackage{todonotes}
\usepackage{adjustbox}
\usepackage[shortlabels]{enumitem}
\usepackage[capitalise]{cleveref}
\usepackage{booktabs}
\usepackage{todonotes}
\usepackage{charter,eulervm}
\usepackage[mathscr]{euscript}

\DeclareFontFamily{OT1}{pzc}{}
\DeclareFontShape{OT1}{pzc}{m}{it}{<-> s * [1.10] pzcmi7t}{}
\DeclareMathAlphabet{\mathpzc}{OT1}{pzc}{m}{it}

\theoremstyle{plain}
\newtheorem{theorem}{Theorem}[section]
\newtheorem*{theorem*}{The Contraction Principle}
\newtheorem{corollary}[theorem]{Corollary}
\newtheorem{proposition}[theorem]{Proposition}

\newtheorem{lemma}[theorem]{Lemma}

\theoremstyle{definition}
\newtheorem*{definition*}{Definition}
\newtheorem{definition}[theorem]{Definition}

\newtheorem{example}[theorem]{Example}

\theoremstyle{remark}
\newtheorem{remark}[theorem]{Remark}
\newtheorem*{notation*}{Notation}

\newtheorem{convention}[theorem]{Convention}

\newcommand{\ignore}[1]{}

\def\s{\sigma}

\def\I{\mathcal{I}}

\def\R{\mathcal{R}}

\def\s{\mathfrak{s}}

\ignore{

\newcommand{\Frm}{\mathbf{Frm}}

\newcommand{\DLat}{\mathbf{DLat}}
\newcommand{\Pries}{\mathbf{Pries}}

\newcommand{\Heyt}{\mathbf{H}}
\newcommand{\proH}{\mathbf{proH}}
\newcommand{\CproH}{\mathbf{CproH}}

\newcommand{\pH}{\mathpzc{p}\mathcal{H}}

\newcommand{\JID}{\mathbf{JID}}
}

\newcommand{\alg}[1]{\mathord{{\mathbf{#1}}}}
\newcommand{\prefixk}[1]{\mathord{\kappa{#1}}}
\newcommand{\prefixC}[1]{\mathord{\mathbf{C}{#1}}}
\newcommand{\DLat}{\alg{DLat}}

\newcommand{\Heyt}{\alg{H}}
\newcommand{\Frm}{\alg{Frm}}
\newcommand{\proH}{\alg{proH}}
\newcommand{\JID}{\alg{JID}}
\newcommand{\kJD}{\prefixk{\alg{JD}}}
\newcommand{\kFrm}{\prefixk\Frm}
\newcommand{\CHeyt}{\prefixC\Heyt}
\newcommand{\CDLat}{\prefixC\DLat}
\newcommand{\kHeyt}{\prefixk\Heyt}
\newcommand{\CkFrm}{\prefixC{\prefixk\Frm}}
\newcommand{\kproH}{\prefixk\proH}

\newcommand{\CproH}{\prefixC{\proH}}
\newcommand{\prokH}{\alg{pro}{\kHeyt}}

\newcommand{\Pries}{\alg{Pries}}
\newcommand{\pH}{\mathpzc{p}\mathcal{H}}

\newcommand{\BL}{\mathcal{BL}}
\newcommand{\M}{\mathcal{DM}}

\newcommand{\st}{\mid}

\title{Degrees of join-distributivity via Bruns-Lakser towers}
\author{G.~Bezhanishvili}
\address{New Mexico State University}
\email{guram@nmsu.edu}

\author{F.~Dashiell Jr}
\address{CECAT, Chapman University}
\email{dashiell@math.ucla.edu }

\author{M.A.~Moshier}
\address{CECAT, Chapman University}
\email{moshier@chapman.edu}

\author{J.~Walters-Wayland}
\address{CECAT, Chapman University}
\email{walterswayland@chapman.edu}

\date{}

\begin{document}


\begin{abstract}
    We utilize the Bruns-Lakser completion to introduce Bruns-Lakser towers of a meet-semilattice. This machinery enables us to develop various hierarchies inside the class of bounded distributive lattices, which measure $\kappa$-degrees of distributivity of bounded distributive lattices and their Dedekind-MacNeille completions. We also use Priestley duality to obtain a dual characterization of the resulting hierarchies. Among other things, this yields a natural generalization of Esakia's representation of Heyting lattices to proHeyting lattices.  
\end{abstract}

\maketitle

\tableofcontents

\section{Introduction}

In the theory of completions, the Dedekind-MacNeille completion occupies a special place. The idea was first described by Dedekind in 1872 (for an English translation see \cite{dedekind_essays_1901}) when he constructed the real numbers as special subsets of the rationals. This was further generalized by MacNeille \cite{macneille_partially_1937} who showed that every poset possesses a completion, which we now call the Dedekind-MacNeille completion.
Shortly after, Funayama \cite{funayama_completion_1944} showed that the Dedekind-MacNeille completion of a distributive lattice may not be distributive.
By contrast, the Dedekind-MacNeille completion of a Heyting lattice is not only distributive, but even satisfies the {\em join infinite distributive law} (JID)
\begin{equation*}
    a\wedge \bigvee S = \bigvee (a \wedge S), \label{JID}
\end{equation*}
and hence is a frame (see, e.g., \cite[p.~111]{johnstone_stone_1982}).

This has motivated the search for conditions under which the Dedekind-MacNeille completion possesses various degrees of distributivity (see, e.g., \cite{cornish_kernels_1978, erne_distributive_1993,ball_dedekind-macneille_2016}).

In \cite{mandelker_relative_1970}, Mandelker introduced an important technique for characterizing distributivity in lattices. Varlet \cite{varlet_relative_1973} extended this to meet-semilattices.
The main ideas of these papers can be summarized beginning with the well-known observation that the Heyting arrow law 
\[
a\wedge b\leq c \mbox{ if and only if } a\leq b\to c
\]
forces a lattice to be distributive. But clearly, the existence of this arrow is not necessary for distributivity.
Mandelker showed that the necessary condition comes down to the shape of the downsets $\langle b,c\rangle = \{x \st x\wedge b\leq c\}$, which he named \emph{relative annihilators}. 
In a Heyting lattice, these are principal downsets.
He showed that a lattice is distributive if and only if each relative annihilator is an ideal.
In short, this explains why the Heyting arrow law forces distributivity, since principal downsets are ideals.
Varlet refined this by showing that a meet-semilattice is distributive if and only if each relative annihilator is up-directed (in meet-semilattices, this is an appropriate replacement for being an ideal). This demonstrates that relative annihilators provide a sort of gauge of the distributivity of a lattice or semilattice.

Re-enforcing this, Ball et al.~\cite[Defn.~3.1]{ball_dedekind-macneille_2016} characterize those meet-semilattices $A$ for which the Dedekind-MacNeille completion $\M A$ is a frame (for a similar result for distributive lattices, see \cite[Thm.~3.4]{cornish_kernels_1978}).
From the foregoing discussion, we know that it is too much to ask for all relative annihilators of $A$ to be principal as that means $A$ is already a Heyting lattice. But as the Funayama example shows, it is also not enough to ask them to be ideals (up-directed downsets). 
The necessary and sufficient condition of \cite{ball_dedekind-macneille_2016} asks that all relative annihilators be \emph{normal}. 
That is, each is an intersection of principal downsets.
This leads to a natural definition of proHeyting semilattices as those meet-semilattices $A$ for which all relative annihilators are normal, so that $\M A$ 
is a frame if and only if $A$ is proHeyting.

The notion of a proHeyting semilattice properly generalizes that of a Heyting lattice \cite[Ex.~3.1]{ball_dedekind-macneille_2016}, and implicitly that of an implicative semilattice \cite{nemitz_implicative_1965,varlet_relative_1973}. Another natural generalization of Heyting lattices is obtained by considering those lattices that satisfy JID for all existing joins. Indeed, for complete lattices, the three notions of Heyting, ProHeyting, and JID are equivalent. In general, they split into the following proper inclusions: 
\[\Heyt \subset \proH \subset \JID.\] 
We note that the last class can conveniently be described as follows: We say that an existing join $\bigvee S$ in a lattice $A$ is {\em distributive}\footnote{Such joins are also known as {\em exact} \cite[p.~531]{ball_notes_2014}.} if $a\wedge \bigvee S = \bigvee (a \wedge S)$ for each $a\in A$. 
Then $A$ is JID provided each existing join in $A$ is distributive. 
The following diagram summarizes the classes of bounded distributive 
lattices mentioned above, where each class is described in the table below the diagram:
 \[\begin{tikzcd}[sep=small]
	\Heyt && \proH && {\JID} && {\DLat}  \\
\CHeyt&& \CproH && {\Frm} && {\CDLat}  \\
	\arrow[hook, from=2-1, to=1-1]
 \arrow[hook, from=2-3, to=1-3]
 \arrow[hook, from=2-5, to=1-5]
 \arrow[hook, from=2-7, to=1-7]
\arrow[Rightarrow, no head, from=2-1, to=2-3]
\arrow[Rightarrow, no head, from=2-3, to=2-5]
\arrow[hook, from=2-5, to=2-7]
\arrow[hook, from=1-1, to=1-3]
\arrow[hook, from=1-3, to=1-5]
\arrow[hook, from=1-5, to=1-7]
\end{tikzcd}
\]

\vspace{-10mm}

\begin{center}
\begin{figure}[!ht]

\begin{tabular}{|l|l|}
\hline
$\DLat$ &Bounded distributive lattices\\
\hline
$\JID$ & Bounded distributive lattices satisfying JID\\
\hline
$\proH$  & Bounded distributive lattices that are proHeyting \\
\hline
$\Heyt$  & Heyting lattices\\
\hline
$\CDLat$&Complete distributive lattices\\
\hline
$\Frm$  & Frames\\
\hline
$\CproH$  & Complete proHeyting lattices\\
\hline
$\CHeyt$  & Complete Heyting lattices\\
\hline
\end{tabular} \label{fig:basic-inclusions}
\end{figure}
\end{center} 

It is well known (see, e.g., \cite[II.1.4]{picado_frames_2012}) that $\CHeyt = \Frm$. Moreover, each complete ProHeyting lattice $A$ is a frame because if $A$ is complete, then $\M A\cong A$. But $\M A$ is a frame because $A$ is ProHeyting. Thus, $A$ is a frame, and hence $\CHeyt = \CproH = \Frm$. 

The other inclusions in the above diagram are all strict.
 The aim of this paper is to provide a more fine-grained analysis of the diagram by introducing proper classes indexed by regular cardinals that in some sense gauge the interplay between various degrees of completeness and distributivity.

One of our key tools, in addition to the relative annihilators, 
is the completion of a meet-semilattice studied by Bruns and Lakser \cite{bruns_injective_1970}, which is called the {\em Bruns-Lakser completion} in \cite{bezhanishvili_dedekind-macneille_2024}. 
The Bruns-Lakser completion $\BL A$ of a semilattice $A$ is always a frame \cite{bruns_injective_1970}. 
For each regular cardinal $\kappa$, we define the $\kappa$-subframe $\BL_\kappa A$ of $\BL A$ generated by $A$, which will be our measurement of distributivity of $\kappa$-joins. We thus arrive at the following {\em Bruns-Lakser tower} of $A$:
\[
A \subseteq \BL_{\omega_0}A \subseteq \BL_{\omega_1}A \subseteq \cdots \subseteq \BL_\kappa A \subseteq \cdots \subseteq \BL A,
\]
where
$\omega_0$ denotes the first infinite cardinal and $\omega_1$ the first uncountable cardinal, both of which are regular. We show that $\BL_\kappa A = A$ if and only if $A$ is a $\kappa$-frame. Therefore, the collapse of the BL-tower at the $\kappa$-level measures the $\kappa$-frameness of $A$.

We also show that, up to isomorphism, $\BL(\M A) = \BL A$, which allows us to generate the following Bruns-Lakser tower of $\M A$ within $\BL A$:
\[
\M A \subseteq \BL_{\omega_0}(\M A) \subseteq \BL_{\omega_1}(\M A) \subseteq \cdots \subseteq \BL_\kappa(\M A) \subseteq \cdots \subseteq \BL A.
\] 
We prove that $\M A$ is $\kappa$-distributive if and only if $\BL_\kappa(\M A) = \M A$, thus yielding a degree of measurement of $\kappa$-distributivity of $\M A$. In particular, $\M A$ is a distributive lattice if and only if $\BL_{\omega_0}(\M A) = \M A$, and it is a frame if and only if $\BL_\kappa(\M A) = \M A$ for all $\kappa$. 
 This yields a generalization of $\proH$ to $\kproH$:  $A\in\kproH$ if and only if $\M A$ is a $\kappa$-frame.

Another important lattice situated between $A$ and $\BL A$ is the {\em proHeyting extension} $\pH A$ of $A$, which is the bounded sublattice of $\BL A$ generated by (the image of) $A$ in $\BL A$. This extension was introduced in \cite{bezhanishvili_dedekind-macneille_2024} and serves as the {\em finitary} version of $\BL A$. It was defined for bounded distributive lattices, but obviously extends to meet-semilattices. We prove that, up to isomorphism, $\BL(\pH A) = \BL A$, allowing us to develop the Bruns-Lakser tower of $\pH A$ within $\BL A$:
\[
\pH A \subseteq \BL_{\omega_0}\pH A \subseteq \BL_{\omega_1}\pH A \subseteq \cdots \subseteq \BL_\kappa\pH A \subseteq \cdots \subseteq \BL A.
\]
 
 We utilize this to define further hierarchies, this time from $\Heyt$ to $\DLat$ and $\proH$, respectively, by generalizing $\Heyt$ to $\kHeyt$ and $\prokH$ (see \cref{def: kappa H}).

For the sake of completeness (no pun intended) we also consider the hierarchies between $\CHeyt$ and $\Heyt$ given by $\kappa$-complete Heyting lattices ($\kappa\CHeyt$, see \cref{def: kappaCH}), between $\Frm$ and $\CDLat$ 
given by $\kappa$-frames that are complete lattices ($\CkFrm$, see \cref{def: CkFrm}), and between $\JID$ and $\DLat$ given by $\kappa$JD-lattices ($\kJD$, see \cref{def: kappa JD 2}). 
The various hierarchies developed in this paper are summarized in \cref{fig:Summary of hierarchies}.

We conclude the paper by utilizing Priestley duality  for distributive lattices to obtain dual characterizations of our various hierarchies. This we do by obtaining a dual characterization of distributive joins in $\M A$ and existing distributive joins in $A$. 
In particular, this yields a characterization of proHeyting lattices 
that generalizes Esakia's representation of Heyting lattices
\cite{esakia_topological_1974}.

\section{Dedekind-MacNeille and Bruns-Lakser completions} \label{sec: prelims}

In this section, we recall the Dedekind-MacNeille and Bruns-Lakser completions---two completions of primary interest to us. We also recall the proHeyting condition and the resulting notion of the proHeyting extension.
Let $A$ be a poset. 
For $S\subseteq A$, we denote by $S^u$ the set of upper bounds and by $S^\ell$ the set of lower bounds of $S$ in $A$. The next definition is well known (see, e.g., \cite[p.~63]{gratzer_lattice_2011}).

\begin{definition}\ \label{defn: normalideal}
\begin{enumerate}
\item A downset $N$ of $A$ is a {\em normal ideal} if $N = N^{u\ell}$ (equivalently, $N$ is an intersection of principal ideals).%\todo{ref}
      \item The {\em Dedekind-MacNeille completion}, or simply the {\em DM-completion}, is the complete lattice of all normal ideals of $A$, denoted \emph{$\M A$.}
     \end{enumerate}
\end{definition}

As mentioned in the introduction, the DM-completion of a distributive lattice may not be distributive. By contrast, the Bruns-Lakser completion of every meet-semilattice is a frame (see \cite{bruns_injective_1970} and \cite{horn_category_1971}). We next recall this important construction. By a {\em meet-semilattice} we mean a poset $A$ in which all finite meets exist. In particular, $A$ has $1$, but $A$ may not have $0$.

\begin{definition}\label{defn: BA}
Let $A$ be a meet-semilattice.
\begin{enumerate}
      \item Let $S \subseteq A$ be such that $\bigvee S$ exists in $A$. We call $S$  \emph{admissible} and $\bigvee S$ a {\em distributive join} provided  
    $a\wedge\bigvee S = \bigvee\{ a\wedge s : s \in S\}$ for each $a\in A$.\footnote{As mentioned in the introduction, such joins are called exact in \cite[p.~531]{ball_notes_2014}.}
   \item A downset $E$ of $A$ is a {\em D-ideal} if $\bigvee S\in E$ for each admissible $S \subseteq E$. 
 \item The {\em Bruns-Lakser completion}, or simply the {\em BL-completion}, is the collection of all D-ideals of $A$, denoted $\BL A$.
     \end{enumerate}
\end{definition}
 
\begin{remark}
   The frame $\BL A$ is the injective hull of $A$ in the category of meet-semilattices \cite{bruns_injective_1970}, whereas $\M A$ is the injective hull of $A$ in the category of posets \cite{banaschewski_categorical_1967}. Both constructions go back to MacNeille; for historical remarks see \cite[Rem.~3.3]{bezhanishvili_semilattice_2024}. An alternative construction of $\BL A$ is obtained from the observation that D-ideals are precisely intersections of relative annihilators (see  \cite[Cor.~1.7]{cornish_kernels_1978} or  \cite[Lem.~3.12]{bezhanishvili_semilattice_2024}).
\end{remark}

Since every normal ideal is a D-ideal, we have $\M A \subseteq \BL A$. In fact, $\M A$ is a sub-meet-semilattice of $\BL A$, but it is not always a sublattice of $\BL A$. The 
inclusion $\BL A \subseteq \M A$ requires the proHeyting condition, which we define next.

\begin{definition}\ \label{defn: ProH}
Let $A$ be a meet-semilattice.
\begin{enumerate}
      \item \cite{mandelker_relative_1970}
 A {\em relative annihilator} of $A$ is a downset of the form 
\[
\langle a,b \rangle = \{ x\in A : a\wedge x\le b\}.
\]
\item Let $\R A$ denote the poset of relative annihilators of $A$.
\item \cite[Def.~3.1]{ball_dedekind-macneille_2016} 
     $A$  is {\em proHeyting} if $\R A \subseteq \M A$; that is, $\langle a,b \rangle$ is a normal ideal for each $a,b\in A$.\footnote{In \cite{ball_dedekind-macneille_2016}, the authors write $b{\downarrow}a$ instead of $\langle a,b \rangle$. They also assume that $b<a$, however this assumption is unnecessary since if $a\le b$ then $\langle a,b \rangle = A = {\downarrow}1$ is principal; and otherwise $a\wedge b < a$ and $\langle a,b \rangle = \langle a,a\wedge b \rangle$.} 
     \end{enumerate}
\end{definition}

Since ${\downarrow}a=\langle 1,a\rangle$ for each $a\in A$, each principal ideal of $A$ is a relative annihilator, and hence the map $a\mapsto{\downarrow}a$ embeds $A$ into $\R A$. Moreover, since ${\downarrow}a\cap{\downarrow}b={\downarrow}(a\wedge b)$ for each $a,b\in A$, this map embeds $A$, as a meet-semilattice, join-densely into $\BL A$.

\begin{convention} \label{conv}
    We  identify
     $A$ with the principal downsets $\{ {\downarrow}a : a\in A\}$ and view 
    $A$ as a subposet of $\R A$ and a sub-meet-semilattice of $\BL A$. 
\end{convention}

\begin{proposition} \label{thm: MA frame}
   \cite[Thm.~6.3]{ball_dedekind-macneille_2016} \label[proposition]{thm: MA frame 1}
    For a meet-semilattice $A$, the following are equivalent:
    \begin{enumerate}[label=\upshape(\arabic*), ref = \thetheorem(\arabic*)]
        \item $A$ is proHeyting.
        \item $\M A$ is a frame.
        \item $\M A = \BL A$.
        \item $\M A$ is a Heyting lattice.
    \end{enumerate} 
\end{proposition} 

\begin{remark}
    For distributive lattices these equivalences are already in Cornish \cite[Thm.~3.4]{cornish_kernels_1978}.
\end{remark}

The finitary version of $\BL A$ was introduced in \cite[Defn.~3.9]{bezhanishvili_dedekind-macneille_2024} as the proHeyting extension of $A$ wherein it was assumed that $A\in\DLat$. We point out that the same construction works for an arbitrary meet-semilattice. 

\begin{definition}\label{def: proH ext}
    The \emph{proHeyting extension} of a meet-semilattice $A$ is the bounded sublattice $\pH A$ of $\BL A$ generated by $\R A$, the relative annihilators of $A$. 
\end{definition}
The following results were proved in \cite[Sec.~3]{bezhanishvili_dedekind-macneille_2024} for  $A\in\DLat$, but the same proofs work in more generality. 

\begin{proposition} Let $A$ be a meet-semilattice. 
    \begin{enumerate}[label=\normalfont(\arabic*), ref = \theproposition(\arabic*)]
        \item $\pH A$ is proHeyting.  
        \item $A = \pH A$ if and only if $A$ is a Heyting lattice. \label[proposition]{prop: pH A 2}
        \item $\pH A =\BL A$ if and only if $\pH A$ is a frame. 
        \item $\M (\pH A) \cong \BL A \cong \BL (\pH A)$. \label[proposition]{prop: pH A 4}
       \end{enumerate}
\end{proposition}
\begin{proof} We refer to 
   \cite{bezhanishvili_dedekind-macneille_2024}:
   for (1) see Cor.~3.13(1),
    for (2) see Prop.~3.11(1),
    for (3) see Prop.~3.11(2), and
    for (4) see Thm.~3.12 and Cor.~3.13(2).
\end{proof}

\begin{definition}  We say that a meet-semilattice $A$ is {\em JID} if every existent join in $A$ is a distributive join.   
       \end{definition}

As an immediate consequence of \cref{thm: MA frame} we obtain:

\begin{proposition}
\cite[Prop.~3.7]{ball_dedekind-macneille_2016}    
\label{cor: ProH implies JID}
If $A$ is proHeyting, then $A$ is JID.    
\end{proposition} 

\begin{proof}
  If $A$ is proH, then $\M A$ is a frame by \cref{thm: MA frame}. Since the embedding of $A$ in $\M A$ preserves all existing joins, we conclude that $A$ must be JID.
\end{proof}

On the other hand, the converse of \cref{cor: ProH implies JID} is not true in general, as we show in \cref{ex: ladder plus N}. For an ordinal $\alpha$, define the {\em $\alpha$-ladder} to be the product of $\alpha$ with the two-element chain $\{0,1\}$. We point out that the $\alpha$-ladder is bounded if and only if $\alpha$ is a successor. We also recall \cite[p.~17]{davey_introduction_2002} that the {\em linear sum} of two lattices $A$ and $B$ is the lattice $A\oplus B$ obtained by putting $B$ on top of $A$.

\begin{example}\label{ex: ladder plus N}
Let $P=(\omega\text{-ladder}) \oplus \omega^{\mathrm{op}}$.
For the sake of avoiding unnecessarily complicated notation, we consider the elements of $P$ as labelled in \cref{fig: countable long ladder+N} with the intended order to be read horizontally:

\begin{center}
    \begin{figure}[!ht]
\begin{tikzpicture}
  
\node  (x1) at (0,0) {$\bullet$};\node  at (0,-.25) {$\scriptstyle{0}$};
  \node  (x2) at (1,0) {$\bullet$};\node  at (1,-.25) {$\scriptstyle{c_0}$};
  \node  (x3) at (2,0) {$\bullet$};\node  at (2,-.25) {$\scriptstyle{c_1}$};
\node(x5) at (4,0) {};
\node (x6) at (5,0){};
\node  (z1) at (1,1) {$\bullet$};\node  at (1,1.25) {$\scriptstyle{b_0}$};
  \node  (z2) at (2,1) {$\bullet$};\node  at (2,1.25) {$\scriptstyle{b_1}$};
  \node  (z3) at (3,1) {$\bullet$};\node  at (3,1.25) {$\scriptstyle{b_2}$};
\node(z5) at (5,1) {};
\node(z6) at (6,1) {};
  \node  (y3) at (7,1) {$\bullet$};\node  at (7,1.25) {$\scriptstyle{a_1}$};
 \node  (y2) at (8,1) {$\bullet$};\node  at (8,1.25) {$\scriptstyle{a_0}$};
  \node  (y1) at (9,1) {$\bullet$};\node  at (9,1.25) {$\scriptstyle{1}$};
  \draw (x1) --(x2) -- (x3); 
   \draw (z1) --(z2) -- (z3); 
   \draw{(x1)--(z1)};
    \draw{(x2)--(z2)};
     \draw{(x3)--(z3)};
     \draw{(y3)--(y2)--(y1)};
 \draw[dashed](x3) --(x5);
  \draw[dashed]{(z3)--(z5)};
 \draw[dashed](z5)--(y3);

\end{tikzpicture}
\caption{$(\omega\text{-ladder})\oplus\omega^{\mathrm{op}}$}
       \label{fig: countable long ladder+N}
    \end{figure}
    \end{center}

    \vspace{-3mm}
    
    As was pointed out in \cite[Ex.~2.4]{bezhanishvili_funayamas_2013}, $P$ is JID since each join in $P$ is finite. On the other hand, $P$ is not proHeyting because the relative annihilator $\langle b_0, 0 \rangle = \bigcup_n {\downarrow}c_n$ is not a normal ideal (since $\langle b_0, 0 \rangle^{ul} = \bigcup_n {\downarrow}b_n$ properly contains $\langle b_0, 0 \rangle$). Therefore, $\M P$ is not a frame by \cref{thm: MA frame}. In fact, $\bigcup_n {\downarrow}b_n$ is the only non-principal normal ideal of $P$ and $\langle b_0,0 \rangle$ is the only D-ideal of $P$ that is not normal. Thus, $\M P$ is obtained by adding $\bigcup_n {\downarrow}b_n$ to $P$, while $\BL P$ by further adding $\langle b_0,0 \rangle$ to $\M P$. Since $\M P\in\DLat$,
    $P$ is an example of a JID lattice such that $\M P$ is still distributive but not JID. In \cref{example: Funayama} we will see that JID does not even guarantee distributivity of the Dedekind-MacNeille completion.
    \end{example}

\section{The BL-tower of a meet-semilattice}

For any regular cardinal $\kappa$, a {\em $\kappa$-set} is a set of cardinality strictly less than $\kappa$ and a {\em $\kappa$-join} is a join of a $\kappa$-set. We have the following obvious generalizations of completeness and JID.

\begin{definition} \label{def: kappa JD} 
    For a meet-semilattice $A$,
    \begin{enumerate}[label=\upshape(\arabic*), ref = \thedefinition(\arabic*)]
    \item $A$ is \emph{$\kappa$-complete} if every $\kappa$-join exists.
     \item $A$ is {\em $\kappa$-join distributive}, abbreviated as $\kappa$JD, if every existent $\kappa$-join is a distributive join. \label[definition]{def: kappa JD 2}
      \item \cite[Defn.~1.1]{madden_kappa-frames_1991} $A$ is a {\em $\kappa$-frame} if it is both $\kappa$-complete and $\kappa$JD. 
   \end{enumerate}     
       \end{definition}
     
   \begin{remark} If $\lambda \leq \kappa$ then $\kappa$-complete/$\kappa$JD implies $\lambda$-complete/$\lambda$JD.  
   Clearly $A$ is a distributive lattice if and only if $A$ is $\omega_0$JD and $A$ is a frame if and only if $A$ is a $\kappa$-frame for $\kappa\ge|A|$.
    \end{remark}  

\begin{definition}
    Let $A$ be a meet-semilattice. 
    \begin{enumerate}
        \item For a regular cardinal $\kappa$, let $\BL_\kappa A$ be the sub-$\kappa$-frame of $\BL A$ generated by $A$. 
    \item We call the following inclusions
    \[
    A \subseteq \BL_{\omega_0} A \subseteq \BL_{\omega_1} A \subseteq \cdots \subseteq \BL_{\kappa} A \subseteq \cdots \BL A
    \]
    the \emph{Bruns-Lakser tower} or simply the \emph{BL-tower} of $A$.
    \item Let $\delta_{\BL}(A)$ be the least cardinal such that $\BL_\kappa A = \BL A$. We call $\delta_{\BL}(A)$ the {\em BL-degree of $A$}. 
    \end{enumerate}
\end{definition}

\begin{remark}\
\begin{enumerate}[label=\upshape(\arabic*), ref = \theremark(\arabic*)]
    \item Since $A$ is a sub-meet-semilattice of $\BL A$, $A$ is $\kappa$-join-dense in $\BL_\kappa A$. \label[remark]{k joins}
    \item $\BL_{\omega_0} A$ is the bounded  distributive sublattice of $\BL A$ generated by $A$, while $\BL_{\omega_1} A$ is the $\sigma$-subframe of $\BL A$ generated by $A$. Thus, the above definition provides a cardinal generalization of 
    \cite[Defn.~3.2~and~4.5]{bezhanishvili_semilattice_2024}, where $\BL_{\omega_0} A$ is denoted by $D A$, $\BL_{\omega_1} A$ by $D\!_\sigma A$, and $\BL A$ by $D\!_\infty A$.
    \item Clearly $\delta_{\BL}(A) \leq |\BL A|$.
\end{enumerate}
\end{remark}

If $A$ and $B$ are meet-semilattices, we write $A\le B$ when $A$ is (isomorphic to) a sub-meet-semilattice of $B$. The next theorem allows us to identify $\BL B$ with $\BL A$, and will play a key r\^ole in what follows. 

\begin{theorem}\label{BL A = BL B}
    Let $A,B$ be meet-semilattices with $A\le B\le \BL A$. Then $\BL B \cong \BL A$.
\end{theorem} 

\begin{proof} Since $\BL A$ is an essential extension of $A$ \cite[Thm.~2]{bruns_injective_1970} and $A\le B\le \BL A$, $\BL A$ is trivially also an essential extension of $B$. 
But $\BL A$ is injective \cite[Cor.~1]{bruns_injective_1970}, so
$\BL A$ is an injective hull of $B$ as defined in \cite[\S 2] {bruns_injective_1970}. However, $\BL B$ is also an injective hull of $B$ \cite[Cor.~2]{bruns_injective_1970}. By uniqueness of the injective hull \cite[p.~34]{balbes_distributive_1974}, we conclude that $\BL B \cong \BL A$.
\end{proof}

\begin{convention}\label{conv: BL A=BL B}
    From now on, for $A\le B\le \BL A$, we will identify $\BL B$ with $\BL A$ and think of $A$ and $B$ as sub-meet-semilattices of $\BL A$. 
\end{convention}

We thus arrive at the following diagram (note that $\delta_{\BL}(A) \geq \delta_{\BL}(B)$):

\begin{figure}[!ht]
\adjustbox{scale=.85}{
\begin{tikzcd}
	B & {\BL_{\omega_0}B} && {\BL_{\kappa}B} && {\BL_{\delta_{\BL}(B)}B} & {\BL_{\delta_{\BL}(A)}B} & \BL B \\
	A & {\BL_{\omega_0}A} && {\BL_{\kappa}A} && {\BL_{\delta_{\BL}(B)}A}& {\BL_{\delta_{\BL}(A)}A} & \BL A
	\arrow[hook, from=1-1, to=1-2]
	\arrow[dashed, hook, from=1-2, to=1-4]
	\arrow[dashed, hook, from=1-4, to=1-6]
	\arrow[Rightarrow, no head, from=1-6, to=1-7]
	\arrow[Rightarrow, no head, from=1-7, to=1-8]
	\arrow[Rightarrow, no head, from=1-8, to=2-8]
	\arrow[hook, from=2-1, to=1-1]
	\arrow[hook, from=2-1, to=2-2]
	\arrow[hook, from=2-2, to=1-2]
	\arrow[dashed, hook, from=2-2, to=2-4]
	\arrow[hook, from=2-4, to=1-4]
	\arrow[dashed, hook, from=2-4, to=2-6]
 \arrow[dashed, hook, from=2-6, to=2-7]
	\arrow[hook, from=2-6, to=1-6]
	\arrow[Rightarrow, no head, from=2-7, to=2-8]
  \arrow[Rightarrow, no head, from=1-7, to=2-7]
\end{tikzcd}
}
 \caption{BL-towers of $A \leq B$}
        \label{BL-towers of A,B}
       \end{figure}

\begin{theorem}\label{sub kappa frame} Let $A,B$ be meet-semilattices with $A\le B\le \BL A$. The following are equivalent: 
\begin{enumerate}[label=\upshape(\arabic*), ref = \thetheorem(\arabic*)]
    \item $B$ is a $\kappa$-frame.
    \item $B$ is a sub-$\kappa$-frame of $\BL A$.
    \item $\BL_\kappa(B) = B$.
\end{enumerate}
\end{theorem}
\begin{proof}
That (3)$\Leftrightarrow$(2) and (2)$\Rightarrow$(1) are straightforward. We prove (1)$\Rightarrow$(2). 
As we pointed out in the proof of \cref{BL A = BL B}, $\BL A$ is an essential extension of $B$. Therefore, by \cite[Thm.~2]{bruns_injective_1970}, the embedding $B\hookrightarrow \BL A$ preserves distributive joins. Thus, since $B$ is a $\kappa$-frame, it is a sub-$\kappa$-frame of $\BL A$.
\end{proof}

As an immediate consequence of \cref{BL A = BL B} we also obtain:

    \begin{proposition}\ \label{BL of DM} For a meet-semilattice $A$ and regular cardinal $\kappa$,
    \begin{enumerate}[label=\upshape(\arabic*), ref = \theproposition(\arabic*)]
        \item  $\BL (\BL_\kappa A) = \BL A$.
        \item  $\BL (\M A) = \BL A$. \label[proposition]{BL of DM 2}
        \item $\BL(\pH A) = \BL A$. \label[proposition]{BL of DM 3}     
    \end{enumerate}
\end{proposition}

\begin{proof}
We have $A\le\BL_\kappa A\le\BL A$, $A\le\M A\le\BL A$, and $A\le\pH A\le\BL A$. Thus, \cref{BL A = BL B} applies in each of these cases, yielding the desired result.
\end{proof}

We conclude this section by giving another characterization of proHeyting meet-semilat\-tices. For this we require the following: 

\begin{lemma}\label{lem: RA}
    Let $A,B$ be meet-semilattices with $A\le B\le \BL A$. Then $\R A$ is order-isomorphic to a subposet of $\R B$.
\end{lemma}

\begin{proof}Define $f : \R A \to \R B$ by $f(\langle a,b\rangle_A) = \langle a,b\rangle_B$. We show that $f$ is well defined and is an order-embedding.  Suppose that $\langle a,b \rangle_A \subseteq \langle c,d \rangle_A $ and $y \in \langle a,b \rangle_B $. Then $a\wedge y \le b$. Therefore, $a\wedge x\le b$ for each $x\in {\downarrow}y\cap A$. Thus, each such $x$ belongs to $\langle c,d \rangle_A$. Since $A$ is join-dense in $\BL A$, $y=\bigvee_{\BL A}({\downarrow}y\cap A)$, so 
    \[
    c\wedge y = c\wedge \bigvee_{\BL A}({\downarrow}y\cap A) = \bigvee_{\BL A}\{ c \wedge x : x \in {\downarrow}y\cap A\} \le \bigvee_{\BL A} \{ c \wedge x : x \in \langle c,d\rangle_A \} \le d, 
    \]
    and hence $y \in \langle c,d\rangle_B$. This proves that $f$ is well defined and is order-preserving. To see that $f$ is order-reflecting, let $\langle a,b \rangle_A \not\subseteq \langle c,d \rangle_A $. Then there is $x\in A$ with $a\wedge x \le b$, but $c\wedge x\not\le d$. Since $A$ is a sub-meet-semilattice of $B$, this gives that $\langle a,b \rangle_B \not\subseteq \langle c,d \rangle_B $, concluding the proof.
\end{proof}

\begin{convention}\label{conv: RA sub RB}
From now on, we will identify $\R A$ with its image in $\R B$.    
\end{convention}

\begin{proposition}\label{cor: pH A pH B}
Let $A,B$ be meet-semilattices with $A\le B\le \BL A$. Then $\pH A$ is 
a bounded sublattice of $\pH B$.
\end{proposition}

\begin{proof}
     By \cref{conv: BL A=BL B}, $\BL B = \BL A$, so $\R B$ is a subposet of $\BL A$, and $\pH B$ is the bounded sublattice of $\BL A$ generated by $\R B$. By \cref{lem: RA,conv: RA sub RB}, $\R A$ is a subposet of $\R B$. Therefore, since $\pH A$ is the bounded sublattice of $\BL A$ generated by $\R A$, it is a bounded sublattice of $\pH B$.
     \end{proof}

\begin{proposition} \label{char of proH}
    For a meet-semilattice $A$, the following are equivalent:
    \begin{enumerate}[label=\upshape(\arabic*), ref = \thecorollary(\arabic*)]
        \item $A$ is proHeyting.
        \item $\BL(\M A) = \M A$.
        \item $\pH(\M A) = \M A$. 
    \end{enumerate}
\end{proposition}
\begin{proof}(1)$\Rightarrow$(2): If $A$ is proHeyting, then $\BL A = \M A$ by \cref{thm: MA frame}. But $\BL A = \BL(\M A)$ by \cref{BL of DM 2}. Thus, $\BL(\M A) = \M A$.

 (2)$\Rightarrow$(3): Since it is always true that $\M A \subseteq\pH(\M A) \subseteq \BL(\M A)$, it follows from $\BL(\M A) = \M A$ that $\pH(\M A) = \M A$.

(3)$\Rightarrow$(1): Since $A \leq \M A$, \cref{lem: RA,conv: RA sub RB} imply that
$\R A \subseteq \R (\M A)$, and hence $\R A \subseteq \pH (\M A)$. By assumption, the latter is $\M A$. Thus, by definition, $A$ is proHeyting.
\end{proof}

\begin{remark}\label{rem: DM vs pH}
Using \cref{conv: BL A=BL B}, it follows from 
\cref{prop: pH A 4,BL of DM 2,char of proH}  
that $A$ proHeyting implies that 
\[
\M(\pH A) =\BL A = \BL(\M A) = \pH(\M A).
\]
Thus, $\M$ and $\pH$ commute if $A$ is proHeyting.  
The converse, however, is not true as will be demonstrated in \cref{example: Funayama}.
\end{remark}

We will be mainly concerned with the BL-towers of $A,\M A,$ and $\pH A$. Since $\M A$ and $\pH A$ are not necessarily comparable, we have the following inclusion relations between these three BL-towers shown in \cref{fig: Bl-towers}. Observe that while $\BL_{\omega_0}\pH A = \pH A$ always, $\BL_{\omega_0} A = A$ if and only if $A\in\DLat$ and $\BL_{\omega_0}(\M A) = \M A$ if and only if ${\M A \in \DLat}$.

\begin{figure}[!ht]
\adjustbox{scale =0.75}{
$
\begin{tikzcd}
{} \\
	\BL (\pH A) && \ \BL(\M A) \\
	& \BL A \\
	{\BL_{\kappa}\pH A} && {\BL_{\kappa}\M A} \\
	& {\BL_{\kappa}A} \\
	{\BL_{\omega_1}\pH A} && {\BL_{\omega_1}\M A} \\
	{\BL_{\omega_0}\pH A} & {\BL_{\omega_1}A} & {\BL_{\omega_0}\M A} \\
	\pH A & {\BL_{\omega_0}A} & \M A \\
	& A
	\arrow[Rightarrow, no head, from=2-1, to=3-2]
	\arrow[Rightarrow, no head, from=3-2, to=2-3]
 \arrow[Rightarrow, no head, from=2-1, to=2-3]
	\arrow[dashed, hook, from=4-1, to=2-1]
	\arrow[dashed, hook, from=4-3, to=2-3]
	\arrow[dashed, hook, from=5-2, to=3-2]
	\arrow[hook', from=5-2, to=4-1]
	\arrow[hook, from=5-2, to=4-3]
	\arrow[dashed, hook, from=6-1, to=4-1]
	\arrow[dashed, hook, from=6-3, to=4-3]
	\arrow[hook, from=7-1, to=6-1]
	\arrow[dashed, hook, from=7-2, to=5-2]
	\arrow[hook', from=7-2, to=6-1]
	\arrow[hook, from=7-2, to=6-3]
	\arrow[hook, from=7-3, to=6-3]
	\arrow[Rightarrow, no head, from=8-1, to=7-1]
	\arrow[hook', from=8-2, to=7-1]
	\arrow[hook, from=8-2, to=7-2]
	\arrow[hook, from=8-2, to=7-3]
	\arrow[hook, from=8-3, to=7-3]
	\arrow[hook', from=9-2, to=8-1]
	\arrow[hook, from=9-2, to=8-2]
	\arrow[hook, from=9-2, to=8-3]
\end{tikzcd}
$
}
\caption{BL-towers of $A$, $\M A$, and $\pH A$}
        \label{fig: Bl-towers}
  
    \end{figure}

\section{Hierarchies for distributive lattices}\label{sec: hierarchies}

In this section, we use the machinery of BL-towers to develop various hierarchies for distributive lattices that in some sense gauge degrees of their join-distributivity. While our hierarchies are for distributive lattices, the results of this section work in the generality of meet-semilattices. 
\subsection{BL-tower of a meet-semilattice: \texorpdfstring{$\kFrm$}--hierarchy}

  Let $\kJD$ and $\kFrm$ be the subclasses of $\DLat$ consisting of all $\kappa$JD lattices and all $\kappa$-frames, respectively. Since finite joins in distributive lattices are distributive,
 $\DLat = \omega_0 \alg{JD} = \omega_0\Frm$. Also, $A \in \Frm$ if and only if $A$ is complete and JID, and more generally, $A \in \kFrm$ if and only if $A$ is $\kappa$-complete and $\kappa$JD. We thus arrive at the following hierarchy. 
\begin{theorem}\ \label{prop: kappa JD}
    \[\begin{tikzcd}[sep=small]
	\JID & {\bigcap_\kappa \kJD} && {\kJD} && \mathord{\omega_1 \alg{JD}} & \mathord{\omega_0\alg{JD}} \\
	&&&&&&& \DLat \\
	\Frm & {\bigcap_\kappa \kFrm} && {\kFrm} && {\mathord{\omega_1\Frm}} & {\mathord{\omega_0\Frm}}
	\arrow[Rightarrow, no head, from=1-1, to=1-2]
	\arrow[dashed, hook, from=1-2, to=1-4]
	\arrow[dashed, hook, from=1-4, to=1-6]
	\arrow[hook, from=1-6, to=1-7]
	\arrow[Rightarrow, no head, from=1-7, to=2-8]
	\arrow[hook, from=3-1, to=1-1]
	\arrow[Rightarrow, no head, from=3-1, to=3-2]
	\arrow[hook, from=3-2, to=1-2]
	\arrow[dashed, hook, from=3-2, to=3-4]
	\arrow[hook, from=3-4, to=1-4]
	\arrow[dashed, hook, from=3-4, to=3-6]
	\arrow[hook, from=3-6, to=1-6]
	\arrow[hook, from=3-6, to=3-7]
	\arrow[Rightarrow, no head, from=3-7, to=2-8]
\end{tikzcd}\]
\end{theorem}

\begin{proof} 
As we already pointed out, $\DLat = \omega_0\alg{JD} = \omega_0\Frm$. Moreover, it is clear that $\lambda\le\kappa$ implies $\kJD\subseteq \lambda\alg{JD}$ and $\kFrm \subseteq \lambda\Frm$. 
    To see that these inclusions are proper when $\lambda<\kappa$, consider $Q_\lambda = \left(\lambda\mbox{-ladder} 
    \right) \oplus 1$. 
     To avoid unnecessarily complicated notation, 
    we label the elements of $Q_\lambda$ as in  \cref{fig:kappa long ladder}, where $\gamma<\lambda$, and again intend the order to be read horizontally.
  
    \begin{center}
    \begin{figure}[!ht]
\begin{tikzpicture}%[scale=.8]
  \node  (x1) at (0,0) {$\bullet$};\node  at (0,-.25) {$\scriptstyle{0}$};
  \node  (x2) at (1,0) {$\bullet$};\node  at (1,-.25) {$\scriptstyle{c_0}$};
  \node  (x3) at (2,0) {$\bullet$};\node  at (2,-.25) {$\scriptstyle{c_1}$};
   \node  (x4) at (4,0) {$\bullet$};\node  at (4,-.25) {$\scriptstyle{c_\gamma}$};
\node(x5) at (6,0) {};
\node  (z1) at (1,1) {$\bullet$};\node  at (1,1.25) {$\scriptstyle{b_0}$};
  \node  (z2) at (2,1) {$\bullet$};\node  at (2,1.25) {$\scriptstyle{b_1}$};
  \node  (z3) at (3,1) {$\bullet$};\node  at (3,1.25) {$\scriptstyle{b_2}$};
  \node  (z4) at (5,1) {$\bullet$};\node  at (5,1.25) {$\scriptstyle{b_\gamma}$};
\node(z5) at (7,1) {};
  \node  (y) at (7.5,1) {$\bullet$};\node  at (7.5,1.25) {$\scriptstyle{1}$};
 \draw (x1) --(x2) -- (x3); 
   \draw (z1) --(z2) -- (z3); 
   \draw{(x1)--(z1)};
    \draw{(x2)--(z2)};
     \draw{(x3)--(z3)};
     \draw{(x4)--(z4)};
  \draw[dashed](x3) --(x5);
  \draw[dashed](z3) --(z5);
  \end{tikzpicture}
\caption{$\lambda$-ladder with top}
        \label{fig:kappa long ladder}
    \end{figure}
    \end{center} 
      Since $\lambda<\kappa$ and $\kappa$ is regular, $b_0\wedge\bigvee \{ c_\gamma : \gamma<\kappa \} = b_0 \wedge 1 = b_0$, but $\bigvee \{ b_0 \wedge c_\gamma : \gamma<\kappa \} = 0$. Thus, 
$Q_\lambda\notin\kJD$.
On the other hand, $Q_\lambda \in \lambda \alg{JD}$ 
since $\lambda$JD holds in the $(\gamma+1)$-ladder for each ordinal
$\gamma<\lambda$. Since $Q_\lambda$ is a complete lattice, we also have that $Q_\lambda\in \lambda\Frm \setminus \kFrm$.
That the inclusion $\kFrm \subseteq \kJD$ is proper can be seen by considering the linear sum $R_\lambda := \lambda\oplus\omega^{\mathrm{op}}$ for $\lambda<\kappa$. We simplify the notation by labelling the elements of the chain $\lambda\oplus\omega^{\mathrm{op}}$ as in \cref{fig:kappa long line} and read the order horizontally.
   
         \begin{center}
    \begin{figure}[!ht]
\begin{tikzpicture}%[scale=.7]
  
\node  (x1) at (0,0) {$\bullet$};\node  at (0,-.25) {$\scriptstyle{0}$};
  \node  (x2) at (1,0) {$\bullet$};\node  at (1,-.25) {$\scriptstyle{a_1}$};
  \node  (x3) at (2,0) {$\bullet$};\node  at (2,-.25) {$\scriptstyle{a_2}$};
\node(x5) at (4,0) {$\bullet$};\node  at (4,-.25) {$\scriptstyle{a_\omega}$};
\node(x6) at (5,0) {$\bullet$};\node  at (5,-.25) {$\scriptstyle{a_{\omega+1}}$};
\node (x7) at (7,0) {};
  \node  (y1) at (11,0) {$\bullet$};\node  at (11,-.25) {$\scriptstyle{1}$};
  \node  (y2) at (10,0) {$\bullet$};\node  at (10,-.25) {$\scriptstyle{b_1}$};
  \node  (y3) at (9,0) {$\bullet$};\node  at (9,-.25) {$\scriptstyle{b_2}$};
  \draw (x1) --(x2) -- (x3); 
  \draw{(x5)--(x6)};
   \draw (y1) --(y2) -- (y3); 
  \draw[dashed](x3) --(x5);
   \draw[dashed](x6)--(x7)--(y3);

\end{tikzpicture}
\caption{Linear sum $\lambda\oplus\omega^{\mathrm{op}}$ 
}
        \label{fig:kappa long line}
    \end{figure}
    \end{center} 

    \vspace{-5mm}  
\noindent Clearly $R_\lambda$ is JID, and hence $R_\lambda \in \kJD$. However, since $R_\lambda$ is not $\kappa$-complete, $R_\lambda\notin\kFrm$.

Finally, it is obvious that $\JID \subseteq \bigcap_\kappa \kJD$. For the reverse inclusion, let $A\in\bigcap_\kappa \kJD$. Then, in particular, $A$ is $\kappa$JD for $\kappa>|A|$, and hence $A\in\JID$. That $\Frm = \bigcap_\kappa \kFrm$ is proved similarly. 
\end{proof}

The next theorem provides a characterization of $\kFrm$ by utilizing the BL-tower of $A$.

 \begin{theorem} \label{thm: kappaJD}
    Let $A$ be a meet-semilattice and $\kappa$ a regular cardinal. Then
    $A\in\kFrm$ if and only if $\BL_\kappa A = A$.
\end{theorem}

\begin{proof}
    Suppose $\BL_\kappa A = A$. Since $\BL_\kappa A$ is a $\kappa$-frame, so is $A$. 
    Conversely, suppose that $A$ is a $\kappa$-frame. By \cite[Cor.~2]{bruns_injective_1970}, the embedding $A \hookrightarrow \BL A$ preserves distributive joins. Therefore, since $A$ is a $\kappa$-frame, 
    it is a sub-$\kappa$-frame of $\BL A$, and hence $\BL_\kappa A = A$.
    \end{proof}

\begin{remark} 
Let $L$ be a frame and $A$ a sub-meet-semilattice that join-generates $L$.
We recall \cite[Defn.~4.3, 4.7]{bezhanishvili_semilattice_2024}
that $L$ is {\em $A$-extremal} (resp.~{\em $A$-basic}) if $A$ 
is a frame (resp.~a $\sigma$-frame). By \cite[Thm.~4.4, 4.8]{bezhanishvili_semilattice_2024}, $L$ is $A$-extremal if and only if $A = \BL A$ (resp.~$L$ is $A$-basic if and only if $A = \BL_{\omega_1} A$).
\cref{thm: kappaJD} is the $\kappa$-generalization of these notions: we could define a frame $L$ to be $A$-$\kappa$-extremal if $A$ is 
a $\kappa$-frame, that is $A = \BL_\kappa A$. (Note that for any meet-semilattice $A$, $\BL A$ is $A$-$\kappa$-extremal for $\kappa \geq \delta_{\BL}(A)$.) Moreover, the $\kappa$ version of $A$-coole (which translated in our terminology is defined to be $\BL_{\omega_1} A = \BL A$) would be that $\BL_{\kappa} A = \BL A$, and this should also be explored.
\end{remark}

\subsection{The BL-tower of the DM-completion: \texorpdfstring{$\kproH$}--hierarchy}
 In the previous theorem, we applied the BL-tower construction to $A$ to obtain a characterization of when $A$ is a $\kappa$-frame. We next apply the BL-tower construction to $\M A$, from which we obtain a characterization of degrees of distributivity of $\M A$. 
This is used to generalize a proHeyting lattice to a $\kappa$-proHeyting lattice, which gives rise to a hierarchy between $\proH$ and the class of those distributive lattices whose Dedekind-MacNeille completion is distributive.

\begin{theorem} \label{cor: kappaJD}
 For a meet-semilattice $A$ and regular cardinal $\kappa$, the following are equivalent: 
   \begin{enumerate}[label=\upshape(\arabic*), ref = \thetheorem(\arabic*)]
       \item $\M A \in \kJD$.
       \item $\M A \in \kFrm$.
       \item $\BL_\kappa (\M A) = \M A$.
   \end{enumerate}    
\end{theorem}
\begin{proof}
The equivalence of (1) and (2) is immediate since $\M A$ is a complete lattice, and the equivalence of (2) and (3) follows from \cref{sub kappa frame}.
\end{proof}

As we pointed out in \cref{char of proH}, $A$ is proHeyting if and only if $\BL(\M A)=\M A$. \cref{cor: kappaJD} affords the following generalization of proHeyting to $\kappa$-proHeyting. 

\begin{definition}\ \label{def: kproH}
    We say that a meet-semilattice $A$ is {\em $\kappa$-proHeyting} if $\BL_\kappa (\M A)= \M A$.
    \end{definition}

\begin{remark}\
\begin{enumerate}[label=\upshape(\arabic*), ref = \theremark(\arabic*)]
    \item Since the inclusion $\M A \subseteq \BL_\kappa (\M A)$ always holds, 
    $A$ is $\kappa$-proHeyting
    if and only if $\BL_\kappa (\M A) \subseteq \M A$. \label[remark]{rem: kproH}
    \item $A$ is proHeyting if and only if $A$ is $\kappa$-proHeyting for  $\kappa \geq \delta_{\BL}(\M A)$. 
\end{enumerate}
\end{remark}

\begin{definition}
    Let $\kproH$ be the subclass 
 of $\DLat$ consisting of $\kappa$-proHeyting lattices.
\end{definition}

We have
the following hierarchies: 
\begin{theorem}\label{ProH hierarchy}
\[\begin{tikzcd}[sep=small]
	\JID & {\bigcap_\kappa \kJD} && {\kJD} && \mathord{\omega_1 \alg{JD}} & \mathord{\omega_0\alg{JD}} & \DLat\\
	&&&&&& \\
	\proH & {\bigcap_\kappa \kproH} && {\kproH} && {\mathord{\omega_1\proH}} & {\mathord{\omega_0\proH}}
	\arrow[Rightarrow, no head, from=1-1, to=1-2]
	\arrow[dashed, hook, from=1-2, to=1-4]
	\arrow[dashed, hook, from=1-4, to=1-6]
	\arrow[hook, from=1-6, to=1-7]
	\arrow[Rightarrow, no head, from=1-7, to=1-8]
	\arrow[hook, from=3-1, to=1-1]
	\arrow[Rightarrow, no head, from=3-1, to=3-2]
	\arrow[hook, from=3-2, to=1-2]
	\arrow[dashed, hook, from=3-2, to=3-4]
	\arrow[hook, from=3-4, to=1-4]
	\arrow[dashed, hook, from=3-4, to=3-6]
	\arrow[hook, from=3-6, to=1-6]
	\arrow[hook, from=3-6, to=3-7]
        \arrow[hook, from=3-7, to=1-7]
\end{tikzcd}\]
\end{theorem}
\begin{proof}
The inclusions in the top row follow from \cref{prop: kappa JD}. To see the inclusions in the bottom row, 
for regular cardinals $\lambda\leq\kappa$, if $A\in\kproH$ then 
$$
\BL_\lambda(\M A)\subseteq \BL_\kappa(\M A) = \M A.
$$
Thus, $A\in\lambda \proH$, and hence $\kproH \subseteq \lambda \proH$. That this inclusion is proper when $\lambda<\kappa$ is witnessed by 
$Q_\lambda=(\lambda\mbox{-ladder})\oplus 1$ (see \cref{fig:kappa long ladder}). Since $Q_\lambda$ is complete, $Q_\lambda = \M Q_\lambda$. 
We have that $\langle b_0,0 \rangle = \bigcup_{\gamma<\lambda}{\downarrow}c_\gamma$ 
is not a normal ideal, so $\langle b_0,0 \rangle \notin \M Q_\lambda$. But it is a $\kappa$-join of principal ideals since $\lambda<\kappa$, so $\langle b_0,0 \rangle \in \BL_\kappa Q_\lambda = \BL_\kappa (\M Q_\lambda)$. Thus, $Q_\lambda\notin\kproH$. 
On the other hand, $Q_\lambda \in \lambda\proH$ by \cref{cor: kappaJD} since $\M Q_\lambda = Q_\lambda \in\lambda \alg{JD}$ (see the proof of \cref{prop: kappa JD}). 

We next prove that $\proH = \bigcap_\kappa \kproH$. First suppose $A\in\proH$. By \cref{char of proH}, $\BL(\M A) \subseteq \M A$. Therefore, $\BL_\kappa (\M A) \subseteq \M A$, and so $A\in\kproH$ for each $\kappa$. For the reverse inclusion, 
$\BL(\M A)=\BL_\kappa(\M A)$ for $\kappa \geq \delta_{\BL}(\M A)$.
Since $A\in\kproH$, $\BL_\kappa(\M A) = \M A$. Thus, $\BL(\M A)=\M A$, and hence 
$A\in\proH$ by \cref{char of proH}.
This finishes the proof of the inclusions in the bottom row. 

It remains to prove that $\kproH \subset \kJD$ for each $\kappa$. If $A\in\kproH$ then $\M A \in \kJD$ by \cref{cor: kappaJD}. Therefore, every $\kappa$-join in $\M A$ is distributive, implying that every existing $\kappa$-join in $A$ is distributive. Thus, $A\in\kJD$. To see that the inclusion is proper for $\kappa>\omega_0$, consider $P_\lambda = (\lambda\text{-ladder}) \oplus \omega^{\mathrm{op}}$ where $\lambda<\kappa$ is infinite. Since every existent nontrivial infinite  $\kappa$-join in $P_\lambda$ belongs to the $(\gamma+1)$-ladder for some ordinal $\gamma<\lambda$, 
we see that $P_\lambda\in\kJD$. But $\BL_\kappa (\M P_\lambda) \not\subseteq \M P_\lambda$ since $\langle b_0, 0\rangle$ belongs to the difference. Therefore, $P_\lambda\notin\kproH$. That $\omega_0\proH$ is properly contained in $\omega_0{\bf JD}$ follows from the existence of distributive lattices $A$ for which $\M A$ is not distributive (see also \cref{DMA remark 1}).
\end{proof}

As an immediate consequence of the above, we obtain: 
 
   \begin{corollary}\ 
    \begin{enumerate}[label=\upshape(\arabic*), ref = \thecorollary(\arabic*)]
        \item $\M A$ is distributive if and only if $A \in \omega_0\proH$. \label[corollary]{DMA distr}
        \item $\M A$ is a frame if and only if $A \in \kproH$ for any $\kappa \geq \delta_{BL}(\M A)$.
    \end{enumerate}
\end{corollary} 

\begin{remark}\ \label{DMA remark}
    \begin{enumerate}[label=\upshape(\arabic*), ref = \theremark(\arabic*)]
    \item Since there exist distributive lattices $A$ such that $\M A$ is not distributive, $\omega_0\proH$ is a proper subclass of $\DLat$. \label[remark]{DMA remark 1}
    \item For $\kappa>\omega_0$, the class $\kproH$ is not comparable with $\kFrm$.  
    To see that $\kproH$ is not contained in $\kFrm$, consider the lattice $R_\lambda$ of  \cref{fig:kappa long line}, where $\lambda<\kappa$ is infinite. As we pointed out in the proof of \cref{prop: kappa JD}, 
    $R_\lambda \notin \kFrm$. On the other hand, since $R_\lambda$ is a Heyting lattice, $\M R_\lambda = \BL R_\lambda$, and hence $R_\lambda \in \kproH$. That $\kFrm$ is not contained in $\kproH$ follows from \cref{example: kappa Funayama}. \label[remark]{kFrm different kproH}
    \item Since $\lambda \le \kappa$ implies $\BL_\lambda (\M A) \subseteq \BL_\kappa (\M A)$, the statement that $A \in \kproH$ for any $\kappa \geq \delta_{\BL}(\M A)$ is equivalent to  $A \in \kproH$ for all $\kappa$. 
    \end{enumerate}
\end{remark}

In the introduction we saw that ${\bf C}\proH = \Frm$. We can now strengthen this by generalizing $\Frm$ to ${\bf C}\kFrm$.

\begin{definition}\label{def: CkFrm}
     Let $\CkFrm = \kFrm \cap \CDLat$. 
\end{definition}

In other words, $A\in\CkFrm$ if and only if $A$ is a complete lattice that is a $\kappa$-frame. 
Observe that $\prefixC{\omega_0\Frm} = {\CDLat}$, 
but $\kFrm$ is properly contained in $\CkFrm$ for $\kappa>\omega_0$ since there exist $\kappa$-frames that are not complete lattices (e.g., the linear sum $R_\kappa = \kappa\oplus\omega^{\mathrm{op}}$). 

\begin{proposition} \label{CproH} 
For each regular cardinal $\kappa$, ${\bf C}\kproH = {\bf C}\kFrm$.
\end{proposition}
\begin{proof}
It is enough to prove that $A$ complete implies that $A\in\kproH$ if and only if $A\in\kFrm$. Let $A$ be complete. Then $\M A = A$. Therefore, by \cref{thm: kappaJD},
\[
A\in \kproH \Longleftrightarrow \BL_\kappa (\M A) = \M A \Longleftrightarrow \BL_\kappa A = A \Longleftrightarrow A\in\kFrm.
\]
\end{proof}

A similar proof to \cref{prop: kappa JD} delivers: 

\begin{theorem}
\[\begin{tikzcd}[sep=small]
	 & {\bigcap_\kappa \kFrm} && {\kFrm} && \mathord{\omega_1 \Frm} & \mathord{\omega_0\Frm} &\DLat \\
	\Frm&&&&&&&\\
	 & {\bigcap_\kappa \CkFrm } && {\CkFrm} && {{\bf C}\mathord{\omega_1}\Frm} & {{\bf C}\mathord{\omega_0}\Frm}&\CDLat
	\arrow[Rightarrow, no head, from=2-1, to=1-2]
        \arrow[Rightarrow, no head, from=2-1, to=3-2]
	\arrow[dashed, hook, from=1-2, to=1-4]
	\arrow[dashed, hook, from=1-4, to=1-6]
	\arrow[hook, from=1-6, to=1-7]
	\arrow[Rightarrow, no head, from=1-7, to=1-8]
        \arrow[Rightarrow, no head, from=3-7, to=3-8]
	\arrow[dashed, hook, from=3-2, to=3-4]
	\arrow[hook, from=3-4, to=1-4]
	\arrow[dashed, hook, from=3-4, to=3-6]
	\arrow[hook, from=3-6, to=1-6]
	\arrow[hook, from=3-6, to=3-7]
        \arrow[hook, from=3-7, to=1-7]
         \arrow[hook, from=3-8, to=1-8]
	\end{tikzcd}\]

\end{theorem}

We have the following characterization of when $A\in\CkFrm$:

\begin{theorem} \label{thm: CkFrm char}
    $A\in \CkFrm$ if and only if $\BL_\kappa(\M A) = A$.
\end{theorem}

\begin{proof}
    Suppose $A\in \CkFrm$. Since $A$ is complete, $A=\M A$. Therefore, since $A$ is a $\kappa$-frame, $\BL_\kappa(\M A)=\BL_\kappa A=A$ by \cref{thm: kappaJD}. 
    Conversely, suppose ${\BL_\kappa(\M A) = A}$. Then $\M A \subseteq A$, so $A$ is complete. Moreover, since $\BL_\kappa(\M A)$ is a $\kappa$-frame, $A$ is a $\kappa$-frame. Thus, $A\in \CkFrm$. 
\end{proof}

\subsection{The BL-tower of the proHeyting extension: \texorpdfstring{$\kappa \CHeyt$}--hierarchy}

We now show how to obtain some other hierarchies of interest by applying the BL-tower construction to $\pH A$. We start with the following, which is an immediate consequence of \cref{sub kappa frame}: 

\begin{proposition} \label{prop: pHA kFrm}
For $A$ a meet-semilattice and $\kappa$ a regular cardinal,
$\BL_\kappa(\pH A) = \pH A$ if and only if $\pH A\in \kFrm$.
\end{proposition} 

While the above proposition does not yield a new hierarchy, one can be obtained by studying 
when $\BL_\kappa(\pH A) = A$. As we will see, this condition characterizes the degree of completeness of Heyting lattices. 

\begin{definition}\label{def: kappaCH}
Let $\kappa \CHeyt$ be the class of $\kappa$-complete Heyting lattices. 
\end{definition}
\begin{proposition} \label{prop: char of kCH}
    $\kappa\CHeyt = \Heyt \cap \kFrm$.
\end{proposition}
\begin{proof}
    Clearly $\Heyt \cap \kFrm \subseteq \kappa \CHeyt$. For the reverse inclusion, it is enough to show that each $A\in \kappa\CHeyt$ is a $\kappa$-frame; that is, every $\kappa$-set $S\subseteq A$ is admissible. Let $a\in A$ and $x$ be an upper bound of $a\wedge S$. Then $a\wedge s\le x$ for each $s\in S$. Since $A$ is a Heyting lattice, $s\le a\to x$ for each $s\in S$, so $\bigvee S\le a\to x$, and hence $a\wedge\bigvee S \le x$. Thus, $S$ is admissible. 
\end{proof}

\begin{theorem} \label{prop: kCH hierarchy}

\[\begin{tikzcd}[sep=small]
	\Frm & {\bigcap_\kappa \kFrm} && {\kFrm} && \mathord{\omega_1 \Frm} & \mathord{\omega_0\Frm} &\DLat \\
	&&&&&&&\\
	\CHeyt & {\bigcap_\kappa \CHeyt} && {\kappa\CHeyt} && {\mathord{\omega_1\CHeyt}} & {\mathord{\omega_0\CHeyt}}&\Heyt
	\arrow[Rightarrow, no head, from=1-1, to=1-2]
	\arrow[dashed, hook, from=1-2, to=1-4]
	\arrow[dashed, hook, from=1-4, to=1-6]
	\arrow[hook, from=1-6, to=1-7]
	\arrow[Rightarrow, no head, from=1-7, to=1-8]
        \arrow[Rightarrow, no head, from=3-7, to=3-8]
	%\arrow[hook, from=3-1, to=1-1]
	\arrow[Rightarrow, no head, from=3-1, to=3-2]
	%\arrow[hook, from=3-2, to=1-2]
	\arrow[dashed, hook, from=3-2, to=3-4]
	\arrow[hook, from=3-4, to=1-4]
	\arrow[dashed, hook, from=3-4, to=3-6]
	\arrow[hook, from=3-6, to=1-6]
	\arrow[hook, from=3-6, to=3-7]
        \arrow[hook, from=3-7, to=1-7]
         \arrow[hook, from=3-8, to=1-8]
	\arrow[Rightarrow, no head, from=3-1, to=1-1]
 \arrow[Rightarrow, no head, from=3-2, to=1-2]
\end{tikzcd}\]

\end{theorem}

\begin{proof}
    The proof is similar to previous proofs, so we only sketch it. 
    It is obvious that $\omega_0\CHeyt = \Heyt$, that $\lambda\le\kappa$ implies $\kappa \CHeyt \subseteq \lambda \CHeyt$, and that ${\CHeyt = \bigcap_\kappa\kappa \CHeyt}$. That the inclusion $\kappa \CHeyt \subseteq \lambda\CHeyt$ is proper for $\lambda<\kappa$ can be seen by considering the linear sum $R_\lambda = \lambda\oplus\omega^{\mathrm{op}}$ of 
   \cref{fig:kappa long line}. 
The chain is $\lambda$-complete, but not $\kappa$-complete since $\lambda<\kappa$, and because it is a Heyting lattice, it witnesses that the  inclusion is proper.
    Finally, $\kappa\CHeyt$ is properly contained in $\kFrm$. Indeed, the lattice $Q_\kappa = \left(\kappa\mbox{-ladder} 
    \right) \oplus 1$ is not a Heyting lattice (not even a proHeyting lattice), so it does not belong to $\kappa \CHeyt$, but $Q_\kappa\in\kFrm$.
\end{proof}

We have the following characterization of when $A\in\kappa \CHeyt$:

\begin{theorem} \label{thm: kCH}
    $A\in \kappa\CHeyt$ if and only if $\BL_\kappa(\pH A) = A$.
\end{theorem}

\begin{proof}
    First suppose $A\in \kappa\CHeyt$. Since $A$ is a Heyting lattice, $A=\pH A$ by \cref{prop: pH A 2}. Therefore, since $A$ is a $\kappa$-frame by \cref{prop: char of kCH}, $\BL_\kappa(\pH A)=\BL_\kappa(A)=A$ by \cref{thm: kappaJD}. Conversely, suppose $\BL_\kappa(\pH A) = A$. Then $\pH A\subseteq A$, so $A\in\Heyt$ by \cref{prop: pH A 2} (since the other inclusion always holds under \cref{conv}). 
     Moreover, since $\BL_\kappa(\pH A)$ is a $\kappa$-frame, $A$ is a $\kappa$-frame. Thus, $A\in \kappa\CHeyt$.
\end{proof}

Before turning to topological duals, 
we develop two more hierarchies, this time from $\Heyt$ to $\DLat$ and $\proH$, respectively. 
Recall that $A\in\DLat$ is a Heyting lattice iff each relative annihilator of $A$ is principal, that is $\R A  \subseteq \BL_{\omega_0} A$ (since $A=\BL_{\omega_0} A$). 
We thus define the hierarchy between $\Heyt$ and $\DLat$ by asking that each relative annihilator of $A$ is in the $\kappa$-portion of the BL-tower of $A$. To obtain a similar hierarchy between $\Heyt$ and $\proH$, we additionally ask that $\R A \subseteq \M A$. 

\begin{definition}\label{def: kappa H}
    We say that 
    \begin{enumerate}[label=\upshape(\arabic*), ref = \thedefinition(\arabic*)]
     \item    $A \in \kHeyt$ provided $\R A \subseteq \BL_\kappa A $, that is each relative annihilator  belongs to $\BL_\kappa A$; 
     \item   $A \in \prokH$ provided $\R A \subseteq \BL_\kappa A \cap \M A$, that is each relative annihilator is normal and belongs to $\BL_\kappa A$.
    \end{enumerate}  
\end{definition}

The following is immediate from the definitions:

\begin{proposition} \label{prop: char of prokH}
    $\prokH = \kHeyt \cap \proH$.
\end{proposition}

\begin{theorem} \label{prop: kH hierarchy}

\[\begin{tikzcd}[sep=small]
	 & {\mathord{\omega_0}\Heyt} & \mathord{\omega_1 \Heyt} &{\kHeyt} &{\bigcup_\kappa \kHeyt} &\DLat \\
	\Heyt&&&&&&&\\
	& {\bf{pro}}{\mathord{\omega_0}\Heyt} & {\bf{pro}}{\mathord{\omega_1} \Heyt} &{\prokH} &{\bigcup_\kappa \prokH} &\proH \\
	\arrow[Rightarrow, no head, from=2-1, to=1-2]
	\arrow[ hook, from=1-2, to=1-3]
	\arrow[dashed, hook, from=1-3, to=1-4]
	\arrow[dashed,hook, from=1-4, to=1-5]
	\arrow[Rightarrow, no head, from=1-5, to=1-6]
    \arrow[Rightarrow, no head, from=3-5, to=3-6]
	\arrow[hook, from=3-3, to=1-3]
	\arrow[ hook, from=3-2, to=3-3]
	\arrow[hook, from=3-3, to=3-4]
	\arrow[dashed, hook, from=3-4, to=3-5]
	\arrow[hook, from=3-5, to=1-5]
        \arrow[hook, from=3-6, to=1-6]
       \arrow[hook, from=3-4, to=1-4]
	\arrow[Rightarrow, no head, from=2-1, to=3-2]
\end{tikzcd}\]

\end{theorem}

\begin{proof}
For any $A \in \DLat$, since $\BL_{\omega_0}A = A$, we see that $A\in\omega_0\Heyt$ if and only if each relative annihilator is principal. Therefore, since principal ideals are normal,  $\omega_0\Heyt = \Heyt = {\bf pro}\omega_0\Heyt$. 
It is also clear that $\lambda\le\kappa$ implies $\lambda\Heyt \subseteq \kHeyt$ and ${\bf pro}\lambda\Heyt \subseteq \prokH$. That the inclusions are proper for $\lambda<\kappa$ is witnessed by the lattice $Q_\lambda$ in \cref{fig:kappa long ladder} (since $\langle b_0,0 \rangle$ is a $\kappa$-join of principal ideals, but not a $\lambda$-join).
The same lattice witnesses that the inclusion $\prokH \subseteq \kHeyt$ is proper (since $\langle b_0,0 \rangle$ is not a normal ideal). 

Any distributive lattice $A$ is in $\kHeyt$ for $\kappa\geq \delta_\BL(A)$, and thus $\bigcup_\kappa \kHeyt = \DLat$. By \cref{prop: char of prokH}, 
$\prokH \subseteq  \proH$ for each $\kappa$. Therefore, $\bigcup_\kappa \prokH \subseteq \proH$. For the reverse inclusion, if $A\in\proH$, then each relative annihilator belongs to $\M A$, which is equal to $\BL A$. But the latter is equal to $\BL_\kappa A$ for $\kappa \geq \delta_{\BL}(A)$, and hence $A\in \prokH$.
\end{proof}

\begin{theorem} \label{thm: kappaH} 
For $A \in \DLat$ and any regular cardinal $\kappa$,
\begin{enumerate}[label=\upshape(\arabic*), ref = \thetheorem(\arabic*)]
    \item  $A\in\kHeyt$ if and only if $\pH A \subseteq \BL_\kappa A $; \label[theorem]{thm: kappaH 1}
    \item  $A\in\prokH$ if and only if $\pH A \subseteq \BL_\kappa A \cap \M A$. \label[theorem]{thm: kappaH 2}
\end{enumerate} 
\end{theorem}
\begin{proof}
(1) This follows from the definition of $\pH A$.

(2) For the forward implication, $\prokH \subseteq \proH$ by \cref{prop: char of prokH}. Therefore, for each $A\in\prokH$, $\M A = \BL A$. Consequently, from $\R A \subseteq \BL_\kappa A \cap \M A$ it follows that  
$\pH A\subseteq \BL_\kappa A \cap \M A$. The reverse implication is obvious. 
\end{proof}

\subsection{Summary of hierarchies}

 The following table summarizes the characterizations, via Bruns-Lakser towers, of the classes of bounded distributive lattices obtained in this paper. The class $\kJD$ is missing from the table because the relevant Bruns-Lakser tower provides a characterization of $\kFrm$ rather than $\kJD$. The latter requires a different tower construction and will be addressed in future work. However, in \cref{JID as Priestley} we do provide a characterization of $\kJD$-lattices in the language of Priestley spaces.

\vspace{2mm}

\begin{center}
    
\begin{tabular}{|l|ll|}
\hline
$A\in\kFrm$  & $\BL_\kappa A = A$ &(\cref{thm: kappaJD})\\
\hline
$A\in\CkFrm$ &$\BL_\kappa(\M A) = A$ &(\cref{thm: CkFrm char}) \\
\hline
$A\in \kappa\CHeyt$& $\BL_\kappa(\pH A) = A$&(\cref{thm: kCH})\\
\hline
$A \in \kproH$ & $\BL_\kappa(\M A) \subseteq \M A$ &(\cref{rem: kproH})\\
\hline
$A\in\kHeyt$ & $\pH A \subseteq \BL_\kappa A $ &(\cref{thm: kappaH 1})\\
\hline
$A\in\prokH$ & $\pH A \subseteq \BL_\kappa A \cap \M A$ &(\cref{thm: kappaH 2})\\
\hline

\end{tabular}

\end{center}

\vspace{2mm}

We point out the following:
    
\begin{itemize}
    \item The classes $\kFrm,\CkFrm,$ and $\kappa\CHeyt$ provide cardinal generalizations of three different attributes of frames. Each is described by the collapse of the $\kappa$-level of the corresponding BL-tower (of $A$, $\M A$, or $\BL A$) down to $A$.   
    \item The $\kHeyt$ and $\prokH$ hierarchies are intrinsically different from the others in that as cardinal increases the class gets bigger.
    \item By \cref{CproH}, there is no need to consider ${\bf C}\kproH$, rendering the non-symmetry of the diagram below.
    \item The classes $\kappa\CDLat$ do not appear in the diagram since the primary interest is in various degrees of distributivity relative to various degrees of completeness.
    \item Developing a similar hierarchy between $\proH$ and $\JID$ requires a different tower, which will be discussed in future work. 
\end{itemize}

The diagram in \cref{fig:Summary of hierarchies} outlines the interactions between
the hierarchies we developed. 
These are obtained by studying the BL-towers of $A$, $\M A$, and $\pH A$, and are often described by the collapse of the $\kappa$-level of the corresponding BL-tower. But there are other possible outcomes of the collapse as well. Some of them are outlined in \cref{prop: other possibilities} below. It might also be worthwhile to study BL-towers of other sub-meet-semilattices of $\BL A$.

\begin{figure}[!ht]
    \begin{center}
  
\adjustbox{scale=.69}{
$
\begin{tikzcd}[column sep = small]
	&&&&& {\color{red}{\JID}} \\
	\\
	&&&&&& {\bigcap_\kappa \kJD} \\
	\\
	&&&&&&& {\kJD} \\
	{\color{red}{\proH}} \\
	&& {\bigcap_\kappa \kproH} &&&&&& {\omega_1 \bf{JD}} \\
	{\bigcup_\kappa \prokH} &&&& {\kproH} \\
	&&&&&& {\omega_1 \proH} &&& {\omega_0 \bf{JD}} \\
	{\prokH} &&&&&&&& {\omega_0 \proH} \\
	&&&&&&&&&& {\color{red}{\DLat}} \\
	{{\bf pro}\omega_1 \Heyt} &&&&&&&& {\bigcup_\kappa \kHeyt} \\
	&&&&&& {\kHeyt} &&& {\omega_0 \Frm} \\
	{\bf pro}\omega_0\Heyt &&&& {\omega_1 \Heyt} &&&&&& {\color{red}{\CDLat}} \\
	&& {\omega_0 \Heyt} &&&&&& {\omega_1 \Frm} \\
	{\color{red}{\Heyt}} &&&&&&&&& {{\bf C}\omega_0 \Frm} \\
	& {\omega_0 \CHeyt} &&&&&& {\kappa \Frm} \\
	&& {\omega_1 \CHeyt} &&&&&& {C\omega_1 \Frm} \\
	&&& {\kappa \CHeyt} &&& {\bigcap_\kappa \kappa \Frm} \\
	&&&& {\bigcap_\kappa \kappa \CHeyt} & {\color{red}{\CHeyt}}={\color{red}{\Frm}} & {\bigcap_\kappa \CkFrm} & {\CkFrm}
	\arrow[Rightarrow, no head, from=1-6, to=3-7]
	\arrow[dashed, hook', from=3-7, to=5-8]
	\arrow[dashed, hook', from=5-8, to=7-9]
	\arrow[hook, from=6-1, to=1-6]
	\arrow[Rightarrow, no head, from=6-1, to=7-3]
	\arrow[hook, from=7-3, to=3-7]
	\arrow[dashed, hook', from=7-3, to=8-5]
	\arrow[hook', from=7-9, to=9-10]
	\arrow[Rightarrow, no head, from=8-1, to=6-1]
	\arrow[hook', from=8-1, to=12-9]
	\arrow[hook, from=8-5, to=5-8]
	\arrow[dashed, hook', from=8-5, to=9-7]
	\arrow[hook, from=9-7, to=7-9]
	\arrow[hook', from=9-7, to=10-9]
	\arrow[dashed, hook, from=10-1, to=8-1]
	\arrow[hook', from=10-1, to=13-7]
	\arrow[hook, from=10-9, to=9-10]
	\arrow[hook', from=10-9, to=11-11]
	\arrow[Rightarrow, no head, from=11-11, to=9-10]
	\arrow[dashed, hook, from=12-1, to=10-1]
	\arrow[hook', from=12-1, to=14-5]
	\arrow[Rightarrow, no head, from=12-9, to=11-11]
	\arrow[dashed, hook, from=13-7, to=12-9]
	\arrow[color={rgb,255:red,235;green,20;blue,31},dotted, Rightarrow, no head, from=13-10, to=9-10]
	\arrow[Rightarrow, no head, from=13-10, to=11-11]
	\arrow[hook, from=14-1, to=12-1]
	\arrow[dashed, hook, from=14-5, to=13-7]
	\arrow[hook, from=14-11, to=11-11]
	\arrow[hook, from=15-3, to=14-5]
	\arrow[color={rgb,255:red,235;green,20;blue,31}, dotted,hook, from=15-9, to=7-9]
	\arrow[hook, from=15-9, to=13-10]
	\arrow[Rightarrow, no head, from=16-1, to=14-1]
	\arrow[Rightarrow, no head, from=16-1, to=15-3]
	\arrow[hook, from=16-10, to=13-10]
	\arrow[Rightarrow, no head, from=16-10, to=14-11]
	\arrow[hook, from=17-2, to=13-10]
	\arrow[Rightarrow, no head, from=17-2, to=16-1]
	\arrow[color={rgb,255:red,235;green,20;blue,31}, dotted, hook, from=17-8, to=5-8]
	\arrow[dashed, hook, from=17-8, to=15-9]
	\arrow[hook, from=18-3, to=15-9]
	\arrow[hook, from=18-3, to=17-2]
	\arrow[hook, from=18-9, to=15-9]
	\arrow[dashed, hook, from=18-9, to=16-10]
	\arrow[hook, from=19-4, to=17-8]
	\arrow[dashed, hook, from=19-4, to=18-3]
	\arrow[dashed, hook, from=19-7, to=17-8]
	\arrow[dashed, hook, from=20-5, to=19-4]
	\arrow[color={rgb,255:red,235;green,20;blue,31}, dotted, hook, from=20-6, to=1-6]
	\arrow[Rightarrow, no head, from=20-6, to=19-7]
	\arrow[Rightarrow, no head, from=20-6, to=20-5]
	\arrow[Rightarrow, no head, from=20-7, to=20-6]
	\arrow[dashed, hook, from=20-7, to=20-8]
	\arrow[hook, from=20-8, to=17-8]
	\arrow[dashed, hook, from=20-8, to=18-9]
\end{tikzcd}
$
}
\end{center}
\caption{Summary of hierarchies}
    \label{fig:Summary of hierarchies}
\end{figure}

\begin{samepage}
\begin{proposition} \label{prop: other possibilities}
    Let $A \in \DLat$ and $\kappa$ be a regular cardinal.
    \begin{enumerate}[label=\upshape(\arabic*), ref = \theproposition(\arabic*)]
    \item $\BL_\kappa A = \pH A$ if and only if $A\in\kHeyt$ and $\pH A$ is a $\kappa$-frame.
    \item $\BL_\kappa(\pH A) = \M A$ 
        if and only if $A\in\proH$ and $\BL_\kappa (\pH A)$ is a frame.
    \item $\BL_\kappa(\M A) = \pH A$ if and only if every normal ideal is in $\pH A$, every relative annihilator is in $\BL_\kappa(\M A)$, and $\pH A$ is a $\kappa$-frame.
    \end{enumerate}
\end{proposition}
\end{samepage}

\begin{proof}
   (1) Suppose $\BL_\kappa A = \pH A$. Then $\pH A$ is a $\kappa$-frame, 
     and $A\in\kHeyt$ by \cref{thm: kappaH 1}. 
    Conversely, suppose $A\in\kHeyt$ and $\pH A$ is a $\kappa$-frame. The former implies that $\pH A \subseteq \BL_\kappa A$. The latter, by \cref{prop: pHA kFrm}, implies that $\BL_\kappa(\pH A)=\pH A$. Therefore, $\BL_\kappa A \subseteq \BL_\kappa(\pH A) = \pH A$, yielding the equality.
 
    (2) Suppose $\BL_\kappa(\pH A) = \M A$. Then $\pH A\subseteq \M A$, so every relative annihilator is a normal ideal, and hence $A\in\proH$. Therefore, $\M A$ is a frame, and so $\BL_\kappa (\pH A)$ is a frame. Conversely, if $A\in\proH$, then $\M A = \BL A$. Also, since $\BL_\kappa (\pH A)$ is a frame, $\BL (\BL_\kappa(\pH A)) = \BL_\kappa(\pH A)$. But $ \BL ( \BL_\kappa(\pH A)) = \BL A$ by \cref{BL A = BL B,{conv: BL A=BL B}} and $\BL A = \M A$ since $A\in\proH$. Thus, $\BL_\kappa(\pH A) = \M A$.

   (3) Suppose $\BL_\kappa(\M A) = \pH A$. Then $\pH A$ is a $\kappa$-frame. Also, $\M A \subseteq \pH A$, so every normal ideal is in $\pH A$, and $\pH A\subseteq \BL_\kappa(\M A)$, so every relative annihilator is in $\BL_\kappa(\M A)$. 
    Conversely, since every relative annihilator is in $\BL_\kappa(\M A)$ and $\BL_\kappa(\M A)$ is a bounded sublattice of $\BL A$, we see that $\pH A \subseteq \BL_\kappa(\M A)$. Also, because every normal ideal is in $\pH A$, we have $\M A \subseteq \pH A$. Therefore, since $\pH A$ is a $\kappa$-frame, \cref{prop: pHA kFrm} implies that $\BL_\kappa(\M A) \subseteq \BL_\kappa(\pH A) = \pH A$, yielding the equality.
\end{proof}

\section{Priestley perspective} \label{sec: duality}

In this final section we utilize Priestley duality for distributive lattices \cite{priestley_representation_1970,priestley_ordered_1972} to provide a different perspective on our hierarchies. Among other things, this will result in a generalization of Esakia representation of Heyting lattices \cite{esakia_topological_1974}.  

We recall that a {\em Priestley space} is a compact space $X$ equipped with a partial order $\le$ satisfying the {\em Priestley separation axiom}: if $x\not\le y$ then there is a clopen upset $U$ such that $x\in U$ and $y\notin U$. It is well known that the topology of each Priestley space is a Stone topology. Let $\Pries$ be the category of Priestley spaces and continuous order-preserving maps. We then have:

\begin{theorem} [Priestley duality] \label{thm: Priestley}
    $\DLat$ is dually equivalent to $\Pries$.
\end{theorem}

In particular, each bounded distributive lattice $A$ is represented as the lattice of clopen upsets of the Priestley space $X$ of prime filters of $A$ ordered by inclusion. The topology on $X$ is given by the basis
\[
\{ \s(a)\setminus\s(b) : a,b\in A \},
\]
where $\s(a) = \{ x \in X : a \in x \}$. The {\em Stone map} $\s:A \to {\sf ClopUp}(X)$ is an isomorphism from $A$ onto the lattice of clopen upsets of $X$, yielding our desired representation. 

Under this isomorphism, the ideals of $A$ are represented as open upsets, normal ideals as Dedekind-MacNeille upsets, and D-ideals as Bruns-Lakser upsets of the Priestley space of $A$. To see this, we recall that each Priestley space $X$ is equipped with two additional topologies: the topology of open upsets and the topology of open downsets. We let $\sf cl$ and $\sf int$ denote the closure and interior of the original Stone topology on $X$, ${\sf cl_1}$ and ${\sf int_1}$ the closure and interior of the open upset topology, and  ${\sf cl_2}$ and ${\sf int_2}$ the closure and interior of the open downset topology. Observe that, for $S\subseteq X$, we have 
\[
{\sf cl}_1(S)={\downarrow}{\sf cl}(S) \quad \mbox{and} \quad {\sf cl}_2(S)={\uparrow}{\sf cl}(S),
\]
and hence 
\[
{\sf int}_1(S) = X \setminus {\downarrow} (X \setminus {\sf int}(S)) \quad \mbox{and} \quad {\sf int}_2(S) = X \setminus {\uparrow} (X \setminus {\sf int}(S)).
\]  

\begin{definition}
    Let $X$ be a Priestley space.
    \begin{enumerate}
        \item Let ${\sf OpUp}(X)$ be the frame of open uspets of $X$.
        \item Call $U\in{\sf OpUp}(X)$ a {\em Dedekind-MacNeille set}, or simply a {\em DM-set}, if ${\sf int_1\,cl_2}(U) = U$. Let ${\sf DM}(X)$ be the complete lattice of DM-sets of $X$.
        \item Call $U\in{\sf OpUp}(X)$ a {\em Bruns-Lakser set}, or simply a {\em BL-set}, if ${\sf int_1\,cl}(U) = U$. Let ${\sf BL}(X)$ be the frame of BL-sets of $X$.
    \end{enumerate}
\end{definition}

\begin{theorem} \label{thm: char of I DM BL}
    Let $A\in{\DLat}$ and $X$ be the Priestley space of $A$.
    \begin{enumerate}[label=\normalfont(\arabic*), ref = \thetheorem(\arabic*)]
    \item \cite{priestley_representation_1970} $\I A \cong {\sf OpUp}(X)$. \label[theorem]{thm: char of I DM BL 1}
    \item \cite[Thm.~3.5]{bezhanishvili_macneille_2004} $\M A \cong {\sf DM}(X)$.
        \label[theorem]{thm: char of I DM BL 2}
        \item \cite[Prop.~6.1, 8.12]{bezhanishvili_proximity_2014} $\BL A \cong {\sf BL}(X)$.
        \label[theorem]{thm: char of I DM BL 3}
    \end{enumerate}
\end{theorem}

\begin{samepage}
    \begin{remark}\
\begin{enumerate}[label=\normalfont(\arabic*), ref = \theremark(\arabic*)]
    \item The isomorphism of \cref{thm: char of I DM BL 1} is obtained by sending each ideal $I$ of $A$ to the open upset $\s[I] := \bigcup \{ \s(a) : a\in I \}$; and the isomorphisms of \cref{thm: char of I DM BL 2,thm: char of I DM BL 3} are the appropriate restrictions.
    \item Finite meets in ${\sf OpUp}(X)$, ${\sf DM}(X)$, and ${\sf BL}(X)$ are finite intersections. The join in ${\sf OpUp}(X)$ is set-theoretic union, in ${\sf DM}(X)$ it is calculated by $\bigvee{\mathcal U} = {\sf int_1\,cl_2}\left(\bigcup {\mathcal U}\right)$ for ${\mathcal U}\subseteq{\sf DM}(X)$ and in ${\sf BL}(X)$ it is calculated by $\bigvee{\mathcal V} = {\sf int_1\,cl}\left(\bigcup {\mathcal V}\right)$ for ${\mathcal V}\subseteq{\sf BL}(X)$. \label[remark]{rem: joins in DM and BL}
\end{enumerate}
\end{remark}
\end{samepage}

We first characterize in the language of Priestley spaces when a join in $\M A$ is distributive. For this we recall that ${\sf ClopUp}(X)$ is a base for open upsets and $\{ F \setminus G : F,G \in {\sf ClopUp}(X) \}$ is a base for open sets.

\begin{theorem}\label{lem: char distr joins}
    Let $X$ be a Priestley space and $\mathcal U \subseteq {\sf DM}(X)$. Then $\bigvee\mathcal U$ is distributive if and only if $\bigvee \mathcal U = {\sf int_1cl}\bigcup\mathcal U$.
  \end{theorem}

\begin{proof}
    First suppose $\bigvee\mathcal U = {\sf int_1cl}\bigcup\mathcal U$. Let $V\in{\sf DM}(X)$ be arbitrary. If $W$ is an upper bound of $\{ V\cap U : U \in\mathcal U\}$ in ${\sf DM}(X)$, then $V\cap U \subseteq W$ for each $U\in\mathcal U$, so $\bigcup\{ V\cap U : U \in \mathcal U \} \subseteq W$, and hence $V\cap \bigcup \mathcal U \subseteq W$. Therefore, $\bigcup \mathcal U \subseteq V^c \cup W$, so ${\sf cl} \bigcup \mathcal U \subseteq V^c \cup {\sf cl}(W)$ because $V^c$ is closed. But then $V\cap {\sf cl} \bigcup \mathcal U \subseteq {\sf cl} W$, so $V\cap {\sf int_1 cl} \bigcup \mathcal U \subseteq {\sf int_1cl}(W) = W$ because $V,W\in{\sf DMUp}(X)$. Thus, $V\cap\bigvee \mathcal U \subseteq W$, and hence $V\cap\bigvee \mathcal U = \bigvee \{ V\cap U : U \in\mathcal U\}$. Consequently, $\bigvee \mathcal U$ is a distributive join.

    Conversely, suppose $\bigvee \mathcal U$ is a distributive join and $\bigvee \mathcal U \not\subseteq {\sf int_1cl}\bigcup\mathcal U$. Since ${\sf ClopUp}(X)$ is a base for open upsets, there is a clopen upset $E$ such that $E \subseteq \bigvee\mathcal U$ but $E\not\subseteq{\sf int_1cl}\bigcup\mathcal U$, and so $E\not\subseteq{\sf cl}\bigcup\mathcal U$. Because $\{ F \setminus G : F,G \in {\sf ClopUp}(X) \}$ is a base for open sets, there are clopen upsets $F,G$ such that $\varnothing\ne F\setminus G\subseteq E$ but $(F\setminus G)\cap \bigcup\mathcal U = \varnothing$. Thus, $F\cap \bigcup\mathcal U\subseteq G$, and so $\bigcup\{ F\cap U : U\in\mathcal U\}\subseteq G$. Since $G\in{\sf DM}(X)$, 
    $\bigvee \{ F\cap U : U\in\mathcal U\}\subseteq G$, so $F\cap\bigvee\mathcal U\subseteq G$ because $\bigvee \mathcal U$ is a distributive join. Consequently, $(F\setminus G)\cap\bigvee\mathcal U=\varnothing$, contradicting that $E \subseteq \bigvee\mathcal U$. 
\end{proof}

The following is an immediate consequence of \cref{lem: char distr joins}:

\begin{corollary} \label{lem: distr joins in MA}
    Let $A\in\DLat$ and $X$ be the Priestley space of $A$. For $\mathcal S\subseteq\M A$, the join $\bigvee \mathcal S$ in $\M A$ is distributive if and only if $\s[\bigvee \mathcal S] = {\sf int_1cl}\bigcup \{ \s[J] : J \in\mathcal S \}$.
\end{corollary}

Next, in analogy with the notion of $\kappa$-cozeros (see, e.g., \cite{madden_kappa-frames_1991, ball_lambda-hollow_2024}), we introduce the notions of $\kappa$-clopen upsets and $\kappa$DM-sets. 

\begin{definition}
    Let $X$ be a Priestley space, $U\in{\sf OpUp}(X)$, and $\kappa$ a regular cardinal.
    \begin{enumerate}
        \item We call $U$ a {\em $\kappa$-clopen upset} if it is a $\kappa$-union of clopen upsets.
        \item We call $U$ a {\em $\kappa$DM-set} if it is a $\kappa$-union of DM-sets.
    \end{enumerate}
\end{definition}

\begin{remark}
    In view of \cref{thm: char of I DM BL}, $\kappa$-clopen upsets correspond to  $\kappa$-joins in $\I(A)$ of principal ideals, while $\kappa$DM-sets to $\kappa$-joins of normal ideals.  
\end{remark}

\begin{theorem}\label{Distr of DM}
Let $A\in{\DLat}$ and $X$ be the Priestley space of $A$.
    \begin{enumerate}[label=\upshape(\arabic*), ref = \thetheorem(\arabic*)]
        \item $\M A \in \kJD$ if and only if for each $\kappa$DM-set $U$ of $X$, ${\sf int_1\,cl_2}(U)={\sf int_1\,cl}(U)$.\label[theorem]{Distr of DM 1}
        \item $\M A \in \JID$ if and only if for each open upset $U$ of $X$, ${\sf int_1\,cl_2}(U)={\sf int_1\,cl}(U)$. \label[theorem]{Distr of DM 2}
    \end{enumerate}
 \end{theorem}

\begin{proof}
    This follows from \cref{lem: distr joins in MA} since
     $\s[\bigvee \mathcal S] = {\sf int_1\,cl_2}\bigcup\{\s[J] : J\in\mathcal S\}$ for each $\mathcal S\subseteq \M A$ (see \cref{rem: joins in DM and BL}).
    \end{proof}

This together with \cref{cor: kappaJD} yields:

\begin{corollary}\ \label{cor: proH}
    \begin{enumerate}[label=\upshape(\arabic*), ref = \thecorollary(\arabic*)]
        \item $A\in\kproH$ if and only if for each $\kappa$DM-set $U$ of $X$, ${\sf int_1\,cl_2}(U)={\sf int_1\,cl}(U)$.\label[corollary]{cor: proH 1}
        \item $A\in\proH$ if and only if for each open upset $U$ of $X$, ${\sf int_1\,cl_2}(U)={\sf int_1\,cl}(U)$. \label[corollary]{cor: proH 2}
    \end{enumerate} 
\end{corollary}

The above characterization of $\proH$ utilizes that $A\in\proH$ if and only if $\M A$ is a frame (see \cref{thm: MA frame}). Below we give another characterization of $\proH$.

\begin{theorem} \label{thm: proH Priestley}
   For $A\in{\DLat}$ and $X$ the Priestley space of $A$, 
    the following are equivalent:
        \begin{enumerate}[label=\normalfont(\arabic*), ref = \thetheorem(\arabic*)]
            \item $A\in\proH$.
            \item For $U,V$ clopen upsets, 
        $X\setminus{\downarrow}(U\setminus V) \in {\sf DM}(X)$. 
            \item For $U$ clopen,  $X\setminus{\downarrow}U \in {\sf DM}(X)$. 
            \label[theorem]{thm: proH Priestley3}
        \end{enumerate}
   \end{theorem}

\begin{proof}
    (1)$\Leftrightarrow$(2): Let $a,b \in A$.  
    We let $U=\s(a)$ and $V=\s(b)$. 
    By \cite[Lem.~3.14]{bezhanishvili_dedekind-macneille_2024},
    $\langle a,b \rangle$
    corresponds to the open upset
    $X \setminus {\downarrow} (U \setminus V)$. 
    Thus, 
    by \cref{thm: char of I DM BL 2},
    $A\in\proH$ if and only if $X \setminus {\downarrow} (U \setminus V)\in{\sf DM}(X)$. 
    
    (2)$\Rightarrow$(3): Let $U\subseteq X$ be clopen. Since the boolean algebra of clopen subsets of $X$ is generated by the clopen upsets of $X$, we have $U=\bigcup_{i=1}^n \left(U_i\setminus V_i\right)$.
     Therefore, since $\downarrow$ commutes with $\bigcup$, we have
    \[
    X\setminus{\downarrow}U = X\setminus{\Big\downarrow}\left( \bigcup_{i=1}^n \left(U_i\setminus V_i\right) \right) = \bigcap_{i=1}^n \left( X\setminus {\downarrow} \left(U_i\setminus V_i\right) \right),
    \]
    which is in ${\sf DM}(X)$ 
    since each $X\setminus {\downarrow}\left(U_i\setminus V_i\right)$ is in ${\sf DM}(X)$ 
    by (2) and ${\sf DM}(X)$ is closed under finite intersections. 
    
    (3)$\Rightarrow$(2): This is obvious.
\end{proof}

\begin{remark} \label{rem: Esakia}
    We recall that a Priestley space $X$ is an {\em Esakia space} if ${\downarrow}U$ is clopen for each clopen $U\subseteq X$. Equivalently, $X$ is an Esakia space provided $X\setminus{\downarrow}U\in{\sf ClopUp}(X)$ for each clopen $U\subseteq X$. Esakia \cite{esakia_topological_1974} proved that each Heyting algebra is represented as the Heyting algebra of clopen upsets of a unique Esakia space. The above theorem is a direct generalization of Esakia's representation to proHeyting lattices.
\end{remark}

Let $A\in\DLat$ and $X$ be the Priestley space of $A$. For $S\subseteq A$, we recall that $\bigvee S$ exists in $A$ if and only if ${\sf cl_2}\left(\bigcup \s[S]\right)$ is clopen, in which case $\s(\bigvee S)={\sf cl_2}\left(\bigcup \s[S]\right)$. Since each open upset is a union of clopen upsets, this yields the following characterization of complete distributive lattices: 

\begin{proposition} \label{prop: EOD} \cite[Sec.~8]{priestley_ordered_1972}
    Let $A\in\DLat$ and $X$ be the Priestley space of $A$. Then $A \in \CDLat$ if and only if ${\sf cl_2}(U)$ is clopen for each $U \in {\sf OpUp}(X)$. 
\end{proposition}

\begin{remark} \label{rem: kappa complete dually}
    The Priestley spaces satisfying the condition in the above proposition are called {\em extremally order disconnected}. The name is motivated by the fact that this condition generalizes extremal disconnectedness for Stone spaces (the condition that characterizes Stone spaces of complete Boolean algebras).
    As an immediate generalization of \cref{prop: EOD}, we obtain: 
    \[
    A \mbox{ is $\kappa$-complete if and only if ${\sf cl_2}(U)$ is clopen for each $\kappa$-clopen upset $U$ of $X$.}
    \]
\end{remark}

To characterize when an existent join in $A$ is distributive, we require the following lemma. 

\begin{lemma}\label{lem: dist joins in A and MA}
    
    For a meet-semilattice $A$,
    if $\bigvee_A S$ exists 
    then $\bigvee_A S$ is distributive if and only if $\bigvee_{\M A} S$ is distributive.
\end{lemma}

\begin{proof}
    The right-to-left implication is clear. For the other implication, suppose that $\bigvee_A S$ is distributive. If $\bigvee_{\M A} S$ is not distributive, then there is $x\in \M A$ such that $x\wedge\bigvee_{\M A} S \not\le \bigvee_{\M A} \{ x \wedge s : s\in S \}$. Since $A$ is join-dense in $\M A$, there is $a\in A$ with $a\le x\wedge\bigvee_{\M A} S$ but $a  \not\le \bigvee_{\M A} \{ x \wedge s : s\in S \}$. The former implies that $a\le x$ and $a\le \bigvee_{\M A} S = \bigvee_A S$. Since $\bigvee_A S$ is distributive, 
    \[
    a = a \wedge \bigvee_A S = \bigvee_A \{ a\wedge s : s \in S \} = \bigvee_{\M A} \{ a\wedge s : s \in S \} \le \bigvee_{\M A} \{ x\wedge s : s \in S \}.
    \]
    This contradicts the latter.
\end{proof}

\begin{theorem}\label{lem: dist join}
    Let $A\in\DLat$ and $X$ be the Priestley space of $A$. Suppose $\bigvee_A S$ exists. Then $\bigvee_A S$ is distributive if and only if 
    $\s(\bigvee_A S)={\sf cl}\bigcup\s[S]$. 
\end{theorem}

\begin{proof}
    Suppose 
    $\bigvee_A S$ exists. By \cref{lem: dist joins in A and MA}, $\bigvee_A S$ is distributive if and only if $\bigvee_{\M A} S$ is distributive. Since $\bigvee_A S = \bigvee_{\M A} S$, the latter is equivalent to $\s(\bigvee_A S) = {\sf int_1cl}\bigcup \s[S]$ by \cref{lem: distr joins in MA}. Because $\s(\bigvee_A S)={\sf cl_2}\bigcup \s[S]$, we obtain that $\bigvee_A S$ is distributive if and only if ${\sf cl_2}\bigcup \s[S] = {\sf int_1cl}\bigcup \s[S]$. This, in turn, is equivalent to ${\sf cl_2}\bigcup \s[S] = {\sf cl}\bigcup \s[S]$ since ${\sf cl_2}\bigcup \s[S]$ is a clopen upset, concluding the proof.
\end{proof}

As a consequence, we obtain the following characterization of $\kJD$ and $\JID$:

\begin{theorem}\label{JID as Priestley}
Let $A\in{\DLat}$ and $X$ be the Priestley space of $A$.
    \begin{enumerate}[label=\upshape(\arabic*), ref = \thetheorem(\arabic*)]
        \item $A \in \kJD$ if and only if for each $\kappa$-open upset $U$, if ${\sf cl_2}(U)$ is clopen, then ${\sf cl_2}(U)={\sf cl}(U)$. \label[theorem]{JID as Priestley 1}
        \item \cite[Thm.~4.6]{bezhanishvili_funayamas_2013} $A \in \JID$ if and only if for each open upset $U$, if ${\sf cl_2}(U)$ is clopen then ${\sf cl_2}(U)={\sf cl}(U)$. \label[theorem]{JID as Priestley 2}
    \end{enumerate}
 \end{theorem}

\begin{proof}
    (1) By definition, $A \in \kJD$ if and only if each existent $\kappa$-join is distributive. By \cref{lem: dist join}, this is equivalent to $\s(\bigvee S)={\sf cl}\bigcup\s[S]$ for each existent $\kappa$-join $\bigvee S$. But $\bigvee S$ exists if and only if ${\sf cl}_2 U$ is clopen for the corresponding $\kappa$-open upset $U=\bigcup\s[S]$. The result follows.

    (2) This follows from (1). 
    \end{proof}

The next theorem provides a characterization of $\kappa$-frames in the language of Priestley spaces, from which we derive the Pultr-Sichler characterization of frames.

\begin{theorem}\label{thm: kFrm}
Let $A\in{\DLat}$ and $X$ be the Priestley space of $A$.
    \begin{enumerate}[label=\upshape(\arabic*), ref = \thetheorem(\arabic*)]
        \item $A \in \kFrm$ if and only if ${\sf cl}(U)$ is a clopen upset for each $\kappa$-clopen upset $U$. \label[theorem]{thm: kFrm 1}
        \item \cite[Thm.~2.3]{pultr_frames_1988} $A \in \Frm$ if and only if ${\sf cl}(U)$ is a clopen upset for each open upset~$U$.\label[theorem]{thm: kFrm 2}
    \end{enumerate}
 \end{theorem}

\begin{proof}
(1) $A \in \kFrm$ if and only if $A$ is $\kappa$-complete and $\kappa$JD. Being $\kappa$-complete is equivalent to ${\sf cl_2}(U)$ being clopen for each $\kappa$-open upset $U$ (see \cref{rem: kappa complete dually}), and being $\kappa$JD is equivalent to ${\sf cl_2}(U)$ clopen implies ${\sf cl_2}(U)={\sf cl}(U)$ for each $\kappa$-open upset $U$ (see \cref{JID as Priestley 1}). These two conditions are equivalent to ${\sf cl}(U)$ being a clopen upset for each $\kappa$-clopen upset $U$.

(2) $A \in \Frm$ if and only if $A$ is complete and JID. Being complete is equivalent to ${\sf cl_2}(U)$ being clopen for each open upset $U$ (see \cref{prop: EOD}), and being JID is equivalent to ${\sf cl_2}(U)$ clopen implies ${\sf cl_2}(U)={\sf cl}(U)$ for each open upset $U$ (see \cref{JID as Priestley 2}). These two conditions are equivalent to ${\sf cl}(U)$ being a clopen upset for each open upset $U$. 
\end{proof}

We next turn our attention to the dual characterization of $\CkFrm$, $\kappa\CHeyt$, $\kHeyt$, and $\prokH$. 

\begin{definition}
    Let $X$ be a Priestley space.
    \begin{enumerate}
        \item We call an open upset $U$ of $X$ a {\em $\kappa${\sf BL}-set} if $U={\sf int_1cl}(V)$ for some $\kappa$-clopen upset $V$.
        \item Let ${\sf BL}_\kappa(X)$ be the poset of $\kappa{\sf BL}$-sets of $X$.
    \end{enumerate}
\end{definition}

\begin{proposition}\label{prop: BLkappa}
    Let $A\in\DLat$ and $X$ be the Priestley space of $A$. Then $\BL_\kappa A$ is isomorphic to ${\sf BL}_\kappa(X)$.
\end{proposition}

\begin{proof}
    By \cref{k joins}, $a\in\BL_\kappa A$ if and only if $a$ is a $\kappa$-join in $\BL A$ from $A$, and $U\in{\sf BL}_\kappa(X)$ if and only if $U$ is a $\kappa$-join in ${\sf BL}(X)$ from ${\sf ClopUp}(X)$. The result follows since $A$ is isomorphic to ${\sf ClopUp}(X)$ and $\BL A$ is isomorphic to ${\sf BL}(X)$.
\end{proof}

\begin{theorem}\label{thm: CkFrm}
Let $A\in{\DLat}$ and $X$ be the Priestley space of $A$.
    \begin{enumerate}[label=\upshape(\arabic*), ref = \thetheorem(\arabic*)]
        \item $A \in\CkFrm$ if and only if $X$ is extremally order disconnected and ${\sf cl}(U)\in{\sf ClopUp}(X)$ for each $\kappa$-clopen upset $U$.\label[theorem]{thm: CkFrm 1}
        \item $A\in\kappa \CHeyt$ if and only if $X$ is an Esakia space and ${\sf cl}(U)\in{\sf ClopUp}(X)$ 
        for each $\kappa$-clopen upset $U$.\label[theorem]{thm: CkFrm 2}
        \item $A\in\kHeyt$ if and only if  $X\setminus{\downarrow}U \in {\sf BL}_\kappa(X)$ for each clopen $U$.\label[theorem]{thm: CkFrm 3}
        \item $A\in\prokH$ if and only if  $X\setminus{\downarrow}U \in {\sf DM}(X) \cap {\sf BL}_\kappa(X)$ for each clopen $U$.\label[theorem]{thm: CkFrm 4} 
         \end{enumerate}
 \end{theorem}

\begin{proof}
    (1) Apply \cref{{prop: EOD},thm: kFrm 1}.

    (2) Apply Esakia's representation of Heyting lattices (see \cref{rem: Esakia}), \cref{prop: char of kCH}, and \cref{thm: kFrm 1}.

    (3) By \cite[Prop.~3.15]{bezhanishvili_dedekind-macneille_2024}, $\pH A$ is isomorphic to the bounded sublattice of ${\sf BL}(X)$ consisting of finite joins of elements of the form $X\setminus{\downarrow}U$ where $U$ is clopen. Therefore, it is enough to apply \cref{thm: kappaH 1,prop: BLkappa}. 
    
    (4) Since $\prokH = \kHeyt \cap \proH$, it is enough to apply (3) and \cref{thm: proH Priestley}.
    \end{proof}

The table summarizes Priestley type characterizations of the main classes in \cref{fig:Summary of hierarchies} that have no cardinality restrictions (we write ${\sf Clop}(X)$ for the base of clopen sets of a Priestley space):

\vspace{2mm}

\begin{center}
\adjustbox{scale=.9}{
    
\begin{tabular}{|l|ll|}
\hline
$A\in\Frm$ & $U\in{\sf OpUp}(X) \Rightarrow {\sf cl}(U)\in{\sf ClopUp}(X)$&(\cref{thm: kFrm 2})  \\
\hline
$A\in\CDLat$ & $U\in{\sf OpUp}(X) \Rightarrow {\sf cl_2}(U)\in{\sf ClopUp}(X)$ &(\cref{prop: EOD}) \\
\hline

\hline
$A\in\Heyt$ & 

$U\in{\sf Clop}(X) \Rightarrow X\setminus{\downarrow}U\in{\sf ClopUp}(X)$&(\cref{rem: Esakia})  \\
  
\hline
$A \in \proH$ & $U\in{\sf Clop}(X) \Rightarrow X\setminus {\downarrow}U \in{\sf DM}(X)$&(\cref{thm: proH Priestley})  \\
\hline
$A \in \JID$ & $U,{\sf cl_2}(U)\in{\sf OpUp}(X) \Rightarrow {\sf cl_2}(U)={\sf cl}(U)$ &(\cref{JID as Priestley 2}) \\
\hline

\end{tabular}
    }
\end{center}

\vspace{2mm}

How the above characterizations split into ones involving cardinality restrictions is summarized below. Note that the first line splits into three cases, the fourth line into two cases, and the second line is not considered since our primary interest is the interaction of various degrees of distributivity and completeness. Nevertheless, as we saw in \cref{rem: kappa complete dually}, the classes $\kappa\CDLat$ have obvious Priestley type characterizations. 

\vspace{2mm}

\begin{center}
\adjustbox{scale=.9}{

\begin{tabular}{|l|ll|}
\hline
$A\in\kFrm$ & $U$ a $\kappa$-clopen upset $\Rightarrow {\sf cl}(U)\in{\sf ClopUp}(X)$ &(\cref{thm: kFrm 1}) \\
\hline
$A\in\CkFrm$ & $A\in\kFrm$ and $U\in{\sf OpUp}(X) \Rightarrow {\sf cl_2}(U)\in{\sf ClopUp}(X)$ &(\cref{thm: CkFrm 1}) \\ 
\hline
$A\in \kappa \CHeyt$ & $A\in\kFrm$ and $X$ is Esakia &(\cref{thm: CkFrm 2})\\
\hline
\hline
$A\in\kHeyt$ & 
$U\in{\sf Clop}(X) \Rightarrow X\setminus{\downarrow}U\in{\sf BL}_\kappa(X)$&(\cref{thm: CkFrm 3})  \\
\hline
$A\in\prokH$ & $U\in{\sf Clop}(X) \Rightarrow X\setminus{\downarrow}U\in{\sf BL}_\kappa(X)\cap{\sf DM}(X)$&(\cref{thm: CkFrm 4})  \\
\hline
\hline
$A \in \kproH$ & $U$ a $\kappa\sf DM$-set  
$\Rightarrow {\sf int_1\,cl_2}(U)={\sf int_1\,cl}(U)$&(\cref{cor: proH 1}) \\
\hline
$A \in \kJD$ & $U$ a $\kappa$-clopen upset and ${\sf cl_2}(U) \in{\sf ClopUp}(X) \Rightarrow {\sf cl_2}(U)={\sf cl}(U)$&(\cref{JID as Priestley 1}) \\
\hline
\end{tabular}
    }
\end{center}

\vspace{2mm}

We conclude the paper with two examples. The first one is the well-known example of Funayama \cite{funayama_completion_1944}, 
which also shows that $\M$ and $\pH$ may commute on a distributive lattice $A$ without $A$ being proH. This fulfills the promise made in \cref{rem: DM vs pH}. The second one provides a $\kappa$ generalization of the Funayama example and shows that $\kFrm$ is not contained in $\kproH$, as promised in \cref{kFrm different kproH}. 

\begin{example}\label{example: Funayama}
Let $A=\omega\oplus\omega^{\mathrm{op}}$, and consider the product $A_1 \times A_2 \times A_3$, where $A_1$ and $A_3$ are both copies of $A$, and $A_2$ is the two-element chain $\{0,1\}$. Following \cite{funayama_completion_1944}, we label $A_1$ as 
\[
b_1 < b_2 < \cdots \cdots < a_2 < a_1
\]
and $A_3$ as 
\[
d_1 < d_2 < \cdots \cdots < c_2 < c_1.
\]
It is well known (and easy to see) that the Priestley space of 
$A_1$ is the chain $X_1$:  
\[
x_1 < x_2 < \cdots \infty_1 < \cdots < y_2 < y_1,
\]
where $x_n = {\uparrow} a_n$ and $y_n = A_1 \setminus {\downarrow} b_n$ for each $n\ge 1$, and $\infty_1=\bigcup_n {\uparrow} a_n$. Therefore,
$\infty_1$ is the only limit point of $X_1$ and $\lim x_n = \infty_1 = \lim y_n$. 

The Priestley space $X_2$ of $A_2$ is the singleton $\{w\}$, where $w = \{1\}$, and the Priestley space $X_3$ of $A_3$ is 
\[
u_1 < u_2 < \cdots \infty_2 \cdots < v_2 < v_1,
\]
where $u_n = {\uparrow} c_n$ and $v_n = A_3 \setminus {\downarrow} d_n$ for each $n\ge 1$, and $\infty_2 = \bigcup_n {\uparrow} c_n$.

The Priestley space of $A_1\times A_2\times A_3$ is (order-homeomorphic to) the disjoint union $Y=X_1 \cup X_2 \cup X_3$; see \cref{fig:Funayama1}. 

The Funayama lattice is the following sublattice $L$ of $A_1\times A_2\times A_3$: 
\begin{eqnarray*}
    L &:=& \{ (a_i,1,c_k) : i,k\ge 1 \}  \cup  \\
&& \{ (b_j,1,c_k) : j,k\ge 1 \} \cup \\
&& \{ (b_j,0,c_k) : j,k\ge 1 \} \cup \\
&& \{ (b_j,0,d_l) : j,l\ge 1 \}. 
\end{eqnarray*}

We show that its Priestley space is order-homeomorphic to the space $X$ depicted in \cref{fig:Funayama2}, which is obtained by extending the order $\le$ of $Y$ so that it is the least partial order containing $\le$ and $\{(\infty_1,w),(w,\infty_2)\}$.\footnote{This was also observed by Rodrigo Almeida (personal communication).} 

  \begin{center}
\begin{figure}[!ht] 
   \adjustbox{scale=.85}{
   
 \begin{subfigure}[b]{0.45\textwidth}
   \begin{tikzpicture}[scale=.8]
   
\node (w) at (0,0) {$\bullet$};\node at (0,-.5){$\scriptstyle{w}$};
\node (inf1) at (-2,-1) {\color{red}{$\bullet$}};\node at (-2.5,-1) {\color{red}{$\scriptstyle{\infty_1}$}};
\node (inf2) at (2,1) {\color{red}{$\bullet$}};\node at (2.5,1) {\color{red}{$\scriptstyle{\infty_2}$}};

\node  (x1) at (-2,-5) {$\bullet$};\node  at (-2.5,-5) {$\scriptstyle{x_1}$};
  \node  (x2) at (-2,-4) {$\bullet$};\node   at (-2.5,-4) {$\scriptstyle{x_2}$};
  \node  (x3) at (-2,-3) {$\bullet$};\node  at (-2.5,-3) {$\scriptstyle{x_3}$};

  \node  (y1) at (-2,3) {$\bullet$};\node  at (-2.5,3) {$\scriptstyle{y_1}$};
  \node  (y2) at (-2,2) {$\bullet$};\node   at (-2.5,2) {$\scriptstyle{y_2}$};
  \node  (y3) at (-2,1) {$\bullet$};\node  at (-2.5,1) {$\scriptstyle{y_3}$};

  \node  (u1) at (2,-3) {$\bullet$};\node  at (2.5,-3) {$\scriptstyle{u_1}$};
  \node  (u2) at (2,-2) {$\bullet$};\node   at (2.5,-2) {$\scriptstyle{u_2}$};
  \node  (u3) at (2,-1) {$\bullet$};\node  at (2.5,-1) {$\scriptstyle{u_3}$};

   \node  (v1) at (2,5) {$\bullet$};\node  at (2.5,5) {$\scriptstyle{v_1}$};
  \node  (v2) at (2,4) {$\bullet$};\node   at (2.5,4) {$\scriptstyle{v_2}$};
  \node  (v3) at (2,3) {$\bullet$};\node  at (2.5,3) {$\scriptstyle{v_3}$};

  \draw (x1) --(x2) -- (x3); 
  \draw[dashed](x3) --(inf1);
   \draw (y1) --(y2) -- (y3); 
  \draw[dashed](y3) --(inf1);
  \draw (u1) --(u2) -- (u3); 
  \draw[dashed](u3) --(inf2);
  \draw (v1) --(v2) -- (v3); 
  \draw[dashed](v3) --(inf2);
 \end{tikzpicture} 
\caption{Priestley space of $A_1\times A_2\times A_3$}
         \label{fig:Funayama1}
    \end{subfigure}
     \begin{subfigure}[b]{0.45\textwidth}
   \begin{tikzpicture}[scale=.8]
   \node (w) at (0,0) {$\bullet$};\node at (0,-.5){$\scriptstyle{w}$};
\node (inf1) at (-2,-1) {\color{red}{$\bullet$}};\node at (-2.5,-1) {\color{red}{$\scriptstyle{\infty_1}$}};
\node (inf2) at (2,1) {\color{red}{$\bullet$}};\node at (2.5,1) {\color{red}{$\scriptstyle{\infty_2}$}};

\node  (x1) at (-2,-5) {$\bullet$};\node  at (-2.5,-5) {$\scriptstyle{x_1}$};
  \node  (x2) at (-2,-4) {$\bullet$};\node   at (-2.5,-4) {$\scriptstyle{x_2}$};
  \node  (x3) at (-2,-3) {$\bullet$};\node  at (-2.5,-3) {$\scriptstyle{x_3}$};

  \node  (y1) at (-2,3) {$\bullet$};\node  at (-2.5,3) {$\scriptstyle{y_1}$};
  \node  (y2) at (-2,2) {$\bullet$};\node   at (-2.5,2) {$\scriptstyle{y_2}$};
  \node  (y3) at (-2,1) {$\bullet$};\node  at (-2.5,1) {$\scriptstyle{y_3}$};

  \node  (u1) at (2,-3) {$\bullet$};\node  at (2.5,-3) {$\scriptstyle{u_1}$};
  \node  (u2) at (2,-2) {$\bullet$};\node   at (2.5,-2) {$\scriptstyle{u_2}$};
  \node  (u3) at (2,-1) {$\bullet$};\node  at (2.5,-1) {$\scriptstyle{u_3}$};

   \node  (v1) at (2,5) {$\bullet$};\node  at (2.5,5) {$\scriptstyle{v_1}$};
  \node  (v2) at (2,4) {$\bullet$};\node   at (2.5,4) {$\scriptstyle{v_2}$};
  \node  (v3) at (2,3) {$\bullet$};\node  at (2.5,3) {$\scriptstyle{v_3}$};

 \draw (inf1) --(w) -- (inf2);
  \draw (x1) --(x2) -- (x3); 
  \draw[dashed](x3) --(inf1);
   \draw (y1) --(y2) -- (y3); 
  \draw[dashed](y3) --(inf1);
  \draw (u1) --(u2) -- (u3); 
  \draw[dashed](u3) --(inf2);
  \draw (v1) --(v2) -- (v3); 
  \draw[dashed](v3) --(inf2);
 \end{tikzpicture} 
 
\caption{Priestley space of Funayama's lattice}
         \label{fig:Funayama2}
    \end{subfigure}
    }
   \caption{} 
        \label{fig:Funayama}
\end{figure}
\end{center}

\vspace{-5mm}

For this it is sufficient to see that $L$ is isomorphic to the lattice of clopen upsets of $X$. 
We have the following clopen upsets of $X$, where $i,j,k,l\ge 1$:
\[
\begin{array}{l}
    {\uparrow}x_i \cup {\uparrow}u_k \\
    {\uparrow}y_j \cup {\uparrow}w \cup {\uparrow}u_k, \ \ {\uparrow}w \cup {\uparrow}u_k, \\
    {\uparrow}y_j \cup {\uparrow}u_k, \ \ {\uparrow}u_k, \\
    {\uparrow}y_j \cup {\uparrow}v_l, \ \ {\uparrow}y_j, \ \ {\uparrow}v_l. \ \  
\end{array}
\]
The isomorphism $\varphi:L\to{\sf ClopUp}(X)$ is then given by 

\begin{eqnarray*}
    \varphi(a_i,1,c_k) &=& {\uparrow}x_i \cup {\uparrow}u_k \\
\varphi(b_j,1,c_k)&=&
\begin{cases}

    {\uparrow}y_j \cup {\uparrow}w \cup {\uparrow}u_k  & \text{if }j>1, \\
    {\uparrow}w \cup {\uparrow}u_k &\mbox{if }j=1 \\
\end{cases}\\
\varphi(b_j,0,c_k)&=&
\begin{cases}

    {\uparrow}y_j \cup {\uparrow}u_k  & \text{if }j>1, \\
   {\uparrow}u_k &\mbox{if }j=1 \\
\end{cases}\\
\varphi(b_j,0,d_l)&=&
\begin{cases}

    {\uparrow}y_j \cup {\uparrow}v_l  & \text{if }j,l>1, \\
    {\uparrow}v_l  & \text{if }j=1,l>1, \\
    {\uparrow}y_j  & \text{if }j>1,l=1, \\
  \varnothing &\mbox{if }j=l=1. \\
\end{cases}
\end{eqnarray*}

As was shown in \cite{funayama_completion_1944}, $\M L$ is not distributive. Therefore, $L$ is not proHeyting by \cref{thm: MA frame}. 
We show that $\M$ and $\pH$ commute. By \cref{prop: pH A 4,conv: BL A=BL B}, $\M(\pH L)=\BL L$. Thus, it is sufficient to show that $\pH(\M L)=\BL L$. 
By \cref{thm: Priestley,thm: char of I DM BL}, we identify $L$ with ${\sf ClopUp}(X)$, $\M L$ with ${\sf DM}(X)$, and $\BL L$ with ${\sf BL}(X)$. Clearly ${\sf Clop}(X)\subseteq {\sf DM}(X) \subseteq {\sf BL}(X)$. 
Let 
\[
V=\bigcup{\uparrow}v_n \mbox{ and } Y=\bigcup{\uparrow}y_n. 
\]
Then the open upsets of $X$ that are not clopen can be described as
\[
V, \ Y, \ V\cup Y, \ V\cup{\uparrow}y_n, Y\cup{\uparrow}v_n, Y\cup{\uparrow}u_n, \mbox{ and } Y\cup\{w\}\cup{\uparrow}u_n \quad (n\ge 1).
\]
Of these, $V,V\cup Y,V\cup{\uparrow}y_n$, and $Y\cup\{w\}\cup{\uparrow}u_n$ are in ${\sf DM}(X)$. Since $Y = X \setminus {\downarrow}v_1$, we see that 
$Y,Y\cup{\uparrow}v_n$, and ${Y\cup{\uparrow}u_n}$ are in $\pH({\sf ClopUp}(X)) \subseteq \pH({\sf DM}(X))$. Thus, $\pH({\sf DM}(X))={\sf BL}(X)$, and hence $\pH(\M L)=\BL L$. 
\end{example}

We point out that the Funayama lattice is JID since every existent join in it is finite. Thus, there exist distributive lattices $A$ such that $A\in\JID$, but $\M A$ is not distributive. As we saw in \cref{DMA distr}, for $\M A$ to be distributive we require that $A\in\omega_0\proH$.

\begin{example}\label{example: kappa Funayama}
Let $\kappa$ be a regular cardinal.
We generalize the Funayama example by replacing the descending countable chains $\{ y_n : n\ge 1 \}$ and $\{ v_n : n \ge 1 \}$ with descending $\kappa$-chains. More precisely, let $X_\kappa$ be obtained from $X$ by replacing $\{ y_n : n\ge 1 \}$ with $\{ y_\gamma : \gamma<\kappa \}$ and $\{ v_n : n\ge 1 \}$ with $\{ v_\gamma : \gamma<\kappa \}$. We view $\{ y_\gamma : \gamma<\kappa \}$ as $\kappa^{\mathrm{op}}$ with the topology that comes from the interval topology on $\kappa$, and view $\{ v_\gamma : \gamma<\kappa \}$ similarly. Then $\infty_1$ is the limit of the chain $\{ y_\gamma : \gamma<\kappa \}$ and $\infty_2$ is the limit of the chain $\{ v_\gamma : \gamma<\kappa \}$, so the topology of $X_\kappa$ is a Stone topology and the Priestley separation axiom is satisfied. Thus, $X_\kappa$ is a Priestley space. 

Let $L_\kappa = {\sf ClopUp}(X_\kappa)$. By \cref{thm: kFrm 1}, $L_\kappa \in \kFrm$. On the other hand, for the same reason the DM-completion of the Funayama lattice is not distributive, neither is $\M(L_\kappa)$. But then $\BL_{\omega_0}(\M L_\kappa) \not\subseteq \M L_\kappa$ by \cref{DMA distr}. Therefore, $\BL_\kappa(\M L_\kappa) \not\subseteq \M L_\kappa$, and hence $L_\kappa\notin\kproH$. 
\end{example}

\bibliographystyle{alpha}

\end{document}